\RequirePackage{fix-cm}
\documentclass{svjour3}                     
\smartqed  
\usepackage{graphicx}
\usepackage{amsmath,amssymb}
\usepackage{newtxtext,newtxmath}

\usepackage[numbered,framed]{matlab-prettifier}
\usepackage{ulem}
\usepackage[english]{babel}
\usepackage{listings}
\usepackage[T1]{fontenc}
\makeatletter
\newcommand{\biggg}{\bBigg@\thr@@}
\newcommand{\Biggg}{\bBigg@{3.5}}
\makeatother
\DeclareMathOperator{\erf}{erf}
\DeclareMathOperator{\sgn}{sgn}
\journalname{}
\begin{document}

\title{A fast convolution method for the fractional Laplacian in $\mathbb{R}$\thanks{This research was partially supported by the research group grant IT1615-22 funded by the Basque Government, and by the grant 
PID2021-126813NB-I00 funded by MCIN/AEI/10.13039/501100011033 and by ``ERDF A way of making Europe''.}
}


\author{Jorge Cayama \and  Carlota Mar\'{\i}a Cuesta \and Francisco de la Hoz  \and Carlos Javier Garc\'{\i}a-Cervera}

\authorrunning{J. Cayama, C.~M. Cuesta, F. de la Hoz, C.~J. Garc\'{\i}a-Cervera} 

\institute{
	J. Cayama \at
	Innovalia Metrology, Parque Tecnol\'ogico de Bizkaia, Edificio 500, Planta 1, 48160 Derio, Spain
	\and
	C.~M. Cuesta \at
	Department of Mathematics, Faculty of Science and Technology, University of the Basque Country \mbox{UPV/EHU}, Barrio Sarriena S/N, 48980 Leioa, Spain
	\and
	F. de la Hoz \at
	Department of Mathematics, Faculty of Science and Technology, University of the Basque Country \mbox{UPV/EHU}, Barrio Sarriena S/N, 48980 Leioa, Spain \\
	\email{francisco.delahoz@ehu.eus}           
	\and
	C.~J. Garc\'{\i}a-Cervera \at
	Department of Mathematics, University of California, Santa Barbara, CA 93106, USA
}

\date{Received: date / Accepted: date}

\maketitle

\begin{abstract}

In this article, we develop a new method to approximate numerically the fractional Laplacian of functions defined on $\mathbb R$, as well as some more general singular integrals. After mapping $\mathbb R$ into a finite interval, we discretize the integral operator using a modified midpoint rule. The result of this procedure can be cast as a discrete convolution, which can be evaluated efficiently using the Fast-Fourier Transform (FFT). The method provides an efficient,  second order accurate,  approximation to the fractional Laplacian, without the need to truncate the domain.

We first prove that the method gives a second-order approximation for the fractional Laplacian and other related singular integrals; then, we detail the implementation of the method using the fast convolution, and give numerical examples that support its efficacy and efficiency; finally, as an example of its applicability to an evolution problem, we employ the method for the discretization of the nonlocal part of the one-dimensional cubic fractional Schr\"odinger equation in the focusing case.


\end{abstract}

\section{Introduction and Preliminaries}

In this paper, we develop a new method that uses the fast convolution to approximate numerically the fractional Laplacian on $\mathbb R$ and some more general singular integrals. Before describing the method, let us focus on the type of operator that we want to approximate.

Among the many possible definitions of the fractional Laplacian (see, e.g., \cite{kwasnicki}, where the author considers $-(-\Delta)^{\alpha/2}$ with $\alpha\in(0,2)$), we choose (as we do in \cite{cayamacuestadelahoz2020,cayamacuestadelahoz2021}):
\begin{equation}
	\label{e:fraclapl}
	(-\Delta)^{\alpha/2}u(x) = c_\alpha\int_{-\infty}^\infty\frac{u(x)-u(x+y)}{|y|^{1+\alpha}}dy,
\end{equation}
where $\alpha\in(0,2)$, and
$$
c_\alpha = \alpha\frac{2^{\alpha-1}\Gamma(1/2+\alpha/2)}{\sqrt{\pi}\Gamma(1-\alpha/2)};
$$
note that all the integrals in this paper that are not absolutely convergent must be understood in the principal value sense.

Henceforth, we concentrate on applying (\ref{e:fraclapl}) to regular functions. In particular, whenever $u$ is a twice continuously differentiable bounded function, i.e., $u\in\mathcal C_b^2(\mathbb{R})$, we can express \eqref{e:fraclapl} as (see \cite{cayamacuestadelahoz2021}):
\begin{equation}
\label{e:fraclapl1}
(-\Delta)^{\alpha/2}u(x) = \frac{c_\alpha}{\alpha}\int_{0}^{\infty}\frac{u_x(x-y) - u_x(x+y)}{y^\alpha}dz,
\end{equation}
from which the following lemma follows.
\begin{lemma}[\cite{cayamacuestadelahoz2021}]
	Let $u\in\mathcal C_b^2(\mathbb{R})$, then, if $\alpha\in [1,2)$, or $\alpha\in(0,1)$ and $\lim_{x\to\pm \infty}u_x(x)=0$, 
	\begin{equation}
		\label{e:fraclapl2}
		(-\Delta)^{\alpha/2}u(x)=\left\{
		\begin{aligned}
			& \frac{1}{\pi}\int_{-\infty}^\infty \frac{u_{x}(y)}{x-y}dy, & & \text{if $\alpha = 1$}, \\
			& \frac{c_{\alpha}}{\alpha(1-\alpha)}\int_{-\infty}^\infty \frac{u_{xx}(y)}{|x-y|^{\alpha-1}}dy, & & \text{if $\alpha \neq 1$}.
		\end{aligned}
		\right.
	\end{equation}
\end{lemma}
In this paper, we work with \eqref{e:fraclapl2}, rather than with \eqref{e:fraclapl}. Note that, in this form, it is obvious that, when $\alpha = 1$, the fractional Laplacian is just the Hilbert transform of the derivative of $u(x)$, i.e., $(-\Delta)^{1/2}u(x) \equiv \mathcal H(u_x(x))$. Since the numerical computation of the Hilbert transform has long been studied (see, e.g., \cite{weideman1995}), we have limited ourselves to considering the case $\alpha\not=1$ in this paper.

One of the main difficulties in the computation of the fractional Laplacian is dealing with the unboundedness of $\mathbb R$. In this paper, we follow the approach in \cite{delahozcuesta2016,cayamacuestadelahoz2020,cayamacuestadelahoz2021}, i.e., we map $\mathbb R$ into a bounded interval, by means of an algebraic map (see \cite{Boyd1987}):
\begin{equation*}
	\xi = \frac{x}{\sqrt{L^2 + x^2}}\in[-1,1] \Longleftrightarrow x = \frac{L\xi}{\sqrt{1 - \xi^2}}\in\mathbb R,
\end{equation*}
with $L > 0$. Then, we consider another change of variable:
\begin{equation}
\label{e:xxis}
x = L\cot(s)\in\mathbb R\Longleftrightarrow \xi = \cos(s)\in[-1,1] \Longleftrightarrow s = \arccos(\xi) \in[0,\pi],
\end{equation}
so we need to express \eqref{e:fraclapl2} in terms of $s\in[0,\pi]$, for which we use $dx = -L\sin^{-2}(s)ds$, together with the following identities (see \cite{Boyd1987}):
\begin{equation*}
u_{x}(x) = - \frac{\sin^{2}(s)}{L}u_{s}(s), \qquad
u_{xx}(x) = \frac{\sin^{4}(s)}{L^{2}}u_{ss}(s) + \frac{2\sin^{3}(s)\cos(s)}{L^{2}}u_{s}(s),
\end{equation*}
where, with some abuse of notation, $u(s) \equiv u(x(s))$. Then, \eqref{e:fraclapl2} becomes (see \cite{cayamacuestadelahoz2021}):
\begin{equation}
\label{e:fraclap0pi2}
(-\Delta)^{\alpha/2}u(s)=\left\{
\begin{aligned}
	& \frac{\sin(s)}{L\pi}\int_{0}^\pi \frac{\sin(\eta)u_s(\eta)}{\sin(s - \eta)}d\eta, & & \text{if $\alpha = 1$}, \\
	& \frac{c_{\alpha}|\sin(s)|^{\alpha-1}}{L^\alpha\alpha(1-\alpha)} \\
	& \quad\cdot\int_0^\pi \frac{\sin^{\alpha}(\eta)(\sin(\eta)u_{ss}(\eta) + 2\cos(\eta)u_{s}(\eta))}{|\sin(s-\eta)|^{\alpha-1}}d\eta, & & \text{if $\alpha \neq 1$},
\end{aligned}
\right.
\end{equation}
where, again abusing notation, we use $(-\Delta)^{\alpha/2}$ for the resulting operator in the variable $s$. The structure of \eqref{e:fraclap0pi2} strongly suggests considering a Fourier series expansion of $u(s)$:
$$
u(s) = \sum_{k = -\infty}^\infty\hat u(k)e^{iks},
$$
and precisely this approach was followed in \cite{cayamacuestadelahoz2021}, where a pseudospectral method was proposed to approximate numerically \eqref{e:fraclap0pi2}, for which a fast and accurate numerical approximation of $(-\Delta)^{\alpha/2}e^{ins}$, with  $n\in\mathbb Z$, was necessary. On the other hand, in \cite{cayamacuestadelahoz2020}, $(-\Delta)^{\alpha/2}e^{i2ns}$ was expressed in terms of the hypergeometric function ${}_2F_1$. In fact, there are more results in this direction, like in \cite{Sheng2020}, where the $n$-dimensional fractional Laplacian of different sets of functions was expressed in terms of hypergeometric functions.

However, in the current paper, we follow a different approach, namely, we consider a second-order modification of the midpoint rule applied to \eqref{e:fraclap0pi2} that, combined with a fast Fourier transform (FFT) based convolution algorithm (see, e.g., \cite{Garcia-Cervera2007}), yields a fast method that allows considering very large numbers of points without the need of evaluating hypergeometric functions. Indeed, some of the commercially available implementations of the hypergeometric functions can be numerically inefficient, and, in some cases even inaccurate (see \cite{cayamacuestadelahoz2020} for a discussion on this, and how the use of variable precision arithmetic may be required).

There is a wide range of publications related to the effectiveness of different methods to solve numerically nonlocal fractional operators. For instance, there are schemes using a finite-difference approximation, such as \cite{TianDu2013,Gao2014,DuoWykZhang2018}, and \cite{HuangOberman2014,HuangOberman2016,MindenYing2018,DuoZhang2019}; the latter schemes are based on singularity subtraction and finite-difference approximation by applying a quadrature rule in a bounded domain. There are also algorithms based on the Caffarelli-Silvestre extension \cite{CaffarelliSilvestre2007,HuLiLi2017,Nochetto2015}, followed by the application of spectral approaches \cite{AcostaBorthagaray2016,MaoKarniadakis2018,ZayernouriKarniadakis2014}. A more recent publication is, for instance, \cite{ChenShen2020}, where a fast spectral Galerkin method using the generalized Laguerre functions for the extension problem is applied. A description of some of these techniques can be found in the recent review article \cite{DElia:2020}.

All these publications have in common the use of truncation in the integration domain (which leads to a natural deterioration of the rate of convergence), or that they solve problems on a bounded domain where different definitions of the fractional Laplacian are not equivalent (see, e.g., the introduction of \cite{bonfortevazquez2016}  and the references therein); therefore, in order to make a fair comparison of the method in this paper, we must do it with those methods that do not truncate the domain. In this regard, with respect to \cite{cayamacuestadelahoz2021}, the method presented here is capable of working with much larger amounts of points (and it does it very efficiently); and with respect to \cite{cayamacuestadelahoz2020}, which relies on the expression of the fractional Laplacian in terms of bases of functions defined on $\mathbb{R}$ using hypergeometric functions (see also \cite{Sheng2020}), the current approach avoids the evaluation of those functions, which is numerically expensive and potentially inaccurate.

This article is organized as follows: In Section \ref{s:quadratureformula}, we introduce the quadrature formula used in this paper, and prove rigorously its order of convergence. This might seem not necessary, given that there are classical results regarding standard quadrature methods for singular integrals (see, e.g., \cite{Lyness1967,waterman64,navot61,kapur.rokhlin:1997}), but we remark that none of these apply directly to the particular quadrature presented here. It might be possible to get the same results by adapting the general method presented in, e.g., \cite{Lyness1967} to our quadrature, but this computation seems more cumbersome than getting the error estimates directly, and thus, we opted for the latter. In Section \ref{s:secondorder}, based on that quadrature formula, we develop a second-order method to approximate numerically the fractional Laplacian for $\alpha\not=1$. In fact, we approximate the more general singular integral:
\begin{equation}
	\label{e:Is0pi}
	I(s) = \int_0^\pi \sin^\beta(\eta)|\sin(\eta - s)|^\gamma f(\eta)d\eta,
\end{equation}
where $\beta > 0$, $\gamma > -1$. The main result in this section is Theorem \ref{theo:int0pibg}, in which we show that the proposed method approximates integral \eqref{e:Is0pi} to second order accuracy. In Section \ref{s:numericalconvo}, we explain the fast convolution algorithm (that we will refer to as fast convolution) applied to \eqref{e:Is0pi}, which yields an efficient evaluation of the numerical quadrature used to  approximate the fractional Laplacian. In Section \ref{s:implementation}, we detail how to implement the method in Matlab \cite{matlab}. In Section \ref{s:numerical}, we carry out the numerical experiments; in particular, we simulate the fractional cubic nonlinear Schr\"odinger equation in the focusing case.

All the simulations have been run in an Apple MacBook Pro (13-inch, 2020, 2.3 GHz Quad-Core Intel Core i7, 32 GB).

\section{A modified midpoint rule}

\label{s:quadratureformula}

Let $f\in\mathcal C^2[a, b]$, $N\in\mathbb{N}$, $h = (b-a) / N$, and denote $x_n = a + hn$, $x_{n+1/2} = a + h(n + 1/2)$ and $x_{n+1} = a + h$ (in general, in this paper, if $g(n)$ is an expression depending on $n$, we denote $x_{g(n)} \equiv a + g(n)h$, or, when $a = 0$,  $x_{g(n)} \equiv g(n)h$); then, the well-known midpoint rule satisfies (see, e.g., \cite{Quarteroni2007}):
\begin{equation}
	\label{e:midpointxnxn1}
	\int_{x_n}^{x_{n+1}} f(x)dx - hf(x_{n+1/2}) = \frac{(x_{n+1} - x_n)^3}{24}f''(\xi_n), \quad \xi_n\in(x_n, x_{n+1}),
\end{equation}
and, hence,
\begin{equation}
	\label{e:midpoint}
	\int_{a}^{b} f(x)dx - h\sum_{n = 0}^{N-1}f(x_{n+1/2}) = \frac{(b-a)^3}{24N^2}f''(\xi), \quad \xi\in(a, b).
\end{equation}
On the other hand, if $f\not\in\mathcal C^2[a, b]$, \eqref{e:midpoint} no longer holds. In particular, the study of the behavior of the midpoint rule and the trapezoidal rule to approximate numerically the integral of $x^\beta f(x)$ over intervals of the form $[0, b]$ has long been studied (see, e.g., \cite{Lyness1967}, for a very detailed exposition); for instance, if $f(0)\not=0$, and $\beta \in (-1, 0)\cup(0,1)$, there exists $K > 0$, such that
\begin{equation}
\label{e:int0bxbf}
\left| \int_{0}^{b} x^\beta f(x)dx  - h\sum_{n = 0}^{N-1}x_{n+1/2}^\beta f(x_{n+1/2})\right| \le \frac{K}{N^{\beta + 1}}, 
\end{equation}
where $h = b / N$, $x_{n+1/2} = (n+1/2)h$. Note that we are replacing here the integrand $x^\beta f(x)$ on each subinterval $[x_n, x_{n+1}]$ in \eqref{e:int0bxbf},  by its value at the middle point of that interval, i.e., $x_{n+1/2}^\beta f(x_{n+1/2})$:
$$
\int_{x_n}^{x_{n+1}}x^\beta f(x)dx \approx \int_{x_n}^{x_{n+1}}x_{n+1/2}^\beta f(x_{n+1/2})dx = hx_{n+1/2}^\beta f(x_{n+1/2}).
$$
However, another option is to evaluate only $f(x)$ at $x = x_{n+1/2}$, and integrate exactly $x^\beta$ over $[x_n, x_{n+1}]$ (when $x_n = 0$, $\beta > -1$ must be imposed):
$$
\int_{x_n}^{x_{n+1}}x^\beta f(x)dx \approx \int_{x_n}^{x_{n+1}}x^\beta f(x_{n+1/2})dx = f(x_{n+1/2})\frac{x_{n+1}^{\beta + 1} - x_n^{\beta + 1}}{\beta + 1}.
$$
In fact, it is possible to consider both positive and negative values of $x_{n}$ and $x_{n+1}$:
\begin{equation}
\label{e:partialmidpoint}
\int_{x_{n}}^{x_{n+1}} |x|^\beta f(x)dx \approx \frac{\sgn(x_{n+1})|x_{n+1}|^{\beta+1} - \sgn(x_n)|x_n|^{\beta+1}}{\beta + 1}f(x_{n+1/2}),
\end{equation}
where $x_{n+1/2} = (x_n + x_{n+1}) / 2$, and if $0\in[x_n, x_{n+1}]$, $\beta > -1$ is required. 

The idea of evaluating just a part of a singular integrand at its middle point, and integrating exactly the remaining expression, is not new (see, e.g., \cite{Garcia-Cervera2007}), but, to the best of our knowledge, it has not been used to approximate numerically the fractional Laplacian. In our case, the reason why \eqref{e:partialmidpoint} is preferable to the standard midpoint rule is because, unlike the latter, the former integrates exactly the singularity; moreover, the quadrature can be written as a discrete convolution, which can be evaluated efficiently with the FFT. Recall that the midpoint rule, the trapezoidal rule and other quadrature rules lose accuracy in the presence of singularities, so by integrating the singularities exactly, there is an improvement with respect to \eqref{e:int0bxbf}, as we show in the following result.
\begin{theorem} \label{theo:intab}
Let $f\in\mathcal C^2[a, b]$, where $0 \le a < b$, and let $\beta \not= -1$ (additionally, $\beta > -1$ is required when $a = 0$), $N \in\mathbb N$, $h = (b - a) / N$, $x_n = a + hn$, etc. Define
\begin{equation}
\label{e:Eint0ab}
E = \int_{a}^{b} x^\beta f(x)dx - \sum_{n = 0}^{N-1}\frac{x_{n+1}^{\beta+1} - x_n^{\beta+1}}{\beta + 1}f(x_{n+1/2}).
\end{equation}
Then, there exists $K > 0$, such that
\begin{equation}
\label{e:int0ab}
|E| \le
\left\{
\begin{aligned}
	& \frac{K}{N^2}, & & \text{if $a > 0$},
	\\
	& \frac{K}{N^2}, & & \text{if $a = 0$ and $\beta \ge 0$},
	\\
	& \frac{K}{N^2}, & & \text{if $a = 0$ and $\beta \in (-1, 0)$ and $f'(0) = 0$},
	\\
	& \frac{K}{N^{2+\beta}}, & & \text{if $a = 0$ and $\beta \in (-1, 0)$ and $f'(0) \not= 0$}.
\end{aligned}
\right.
\end{equation}
\end{theorem}

In the rest of this section, we will prove Theorem \ref{theo:intab}, for which we need a number of auxiliary results.

\subsection{Some auxiliary results}

Suppose that $\beta\not\in \mathbb N \cup \{0\}$; then, according to the well-known Newton's generalized binomial theorem,
\begin{equation}
	\label{e:generbinom}
	(1 + x)^{\beta} = \sum_{n = 0}^{\infty}\binom{\beta}{n}x^n,
\end{equation}
where $0^0\equiv 1$. Note that, when $\beta\in \mathbb N \cup \{0\}$, we have the standard binomial theorem. We recall two important properties of the binomial coefficients, for $n\in\mathbb N$:
$$
\binom{\beta}{n} \equiv \frac\beta n\binom{\beta - 1}{n - 1}, \qquad \binom{\beta}{n} \equiv \binom{\beta-1}{n} + \binom{\beta-1}{n - 1}.
$$
On the other hand, in \eqref{e:generbinom}, it is immediate to check that the infinite sum is absolutely convergent for $|x| < 1$, and, if additionally $\beta > 0$, then it is also absolutely convergent when $|x| = 1$; this can be checked, e.g., by applying the well-known D'Alembert and Raabe criteria, respectively. In this paper, all the infinite sums of functions appearing are absolutely convergent for the considered parameters and values of $x$ (which can be checked by using those criteria or any other), and hence, their terms can be freely reordered, etc. In particular, we will use the following consequence of Newton's generalized binomial theorem:
	\begin{align}
		\label{e:1xb-1-xb}
		\frac{(1 + x)^{\beta} - (1 - x)^{\beta}}\beta & = \frac{1}{\beta}\left[\sum_{n = 0}^{\infty}\binom{\beta}{n}x^n - \sum_{n = 0}^{\infty}\binom{\beta}{n}(-1)^nx^n\right]
= \frac{2}{\beta}\sum_{n = 0}^{\infty}\binom{\beta}{2n+1}x^{2n+1} \cr
& = \sum_{n = 0}^{\infty}\frac{2}{2n+1}\binom{\beta-1}{2n}x^{2n+1}.
	\end{align}
We also need the following property of the generalized binomial coefficients.
\begin{lemma} \label{lemma:genbincoef}
	Let $m,l\in \mathbb N \cup \{0\}$, then
	\begin{equation}
		\label{e:genbincoef}
		\sum_{k = 0}^{m}(-1)^{m-k}\binom{m}{k}\binom{\beta + k}{l} = 
		\left\{\begin{aligned}
			& \binom{\beta}{l - m}, & & \text{if $l \ge m$},
			\cr
			& 0, & & \text{if $l < m$}.
		\end{aligned}
		\right.
	\end{equation}
\end{lemma}
\begin{proof} We use an induction argument. When $m = 0$, \eqref{e:genbincoef} is trivially true:
	$$
	\sum_{k = 0}^{0}(-1)^{0-k}\binom{0}{k}\binom{\beta + k}{l} = \binom{\beta}{l}.
	$$
	Suppose now that \eqref{e:genbincoef} holds until a certain $m\in\mathbb N$. We will prove it for $m = m + 1$:
	\begin{align*}
		& \sum_{k = 0}^{m+1}(-1)^{m+1-k}\binom{m+1}{k}\binom{\beta + k}{l} = \sum_{k = 1}^{m+1}(-1)^{m+1-k}\binom{m}{k-1}\binom{\beta + k}{l}
		\cr
		& \quad \qquad + \sum_{k = 0}^{m}(-1)^{m+1-k}\binom{m}{k}\binom{\beta + k}{l}
		\cr
		& \quad = \sum_{k = 0}^{m}(-1)^{m-k}\binom{m}{k}\binom{(\beta + 1) + k}{l} - \sum_{k = 0}^{m}(-1)^{m-k}\binom{m}{k}\binom{\beta + k}{l}
		\cr
		& \quad =
		\left\{
		\begin{aligned}
			&\binom{\beta + 1}{l - m} - \binom{\beta}{l - m} = \binom{\beta}{l - (m + 1)}, & & \text{if $l \ge m + 1$},
			\cr
			&\binom{\beta + 1}{l - m} - \binom{\beta}{l - m} = \binom{\beta + 1}{0} - \binom{\beta}{0} = 1 - 1 = 0, & & \text{if $l = m$},
			\cr
			&0 - 0 = 0, & & \text{if $l < m$},
		\end{aligned}
		\right.
	\end{align*}
	where we have divided the case $l < m + 1$ in two cases, $l = m$ and $l < m$. \qed
\end{proof}

At its turn, Lemma \ref{lemma:genbincoef} is needed to prove the following result.
\begin{lemma} \label{lemma:intann1xbfx}
	Let $\beta \not= -1$, $m\in\mathbb N\cup\{0\}$, $a \ge 0$, $h > 0$, $n\in\mathbb N\cup\{0\}$ (additionally, $\beta > -1$ is required, when both $a = 0$ and $n = 0$), $x_{n} = a + nh$, etc., then
	\begin{align}
		\label{e:intann1xbfx}
		\int_{x_n}^{x_{n+1}}x^\beta(x - x_{n+1/2})^mdx
 = \left\{
		\begin{aligned}
			& \sum_{l = 0}^\infty\frac{h^{2l + m + 1}}{2^{2l + m}(2l + m + 1)}\binom{\beta}{2l}x_{n+1/2}^{\beta - 2l}, & & m\equiv0\bmod2,
			\cr
			& \sum_{l = 0}^\infty\frac{h^{2l + m + 2}}{2^{2l + m + 1}(2l + m + 2)}\binom{\beta}{2l + 1}x_{n+1/2}^{\beta - 2l - 1}, & & m\equiv1\bmod2.
		\end{aligned}
		\right.
	\end{align}
	
\end{lemma}

\begin{proof}	
	In order to compute the integral, we expand $(x - x_{n+1/2})^m$ by means of Newton's binomial:
	\begin{align*}
		\int_{x_n}^{x_{n+1}} & x^\beta (x - x_{n+1/2})^mdx = \sum_{k = 0}^{m}\binom{m}{k}(-x_{n+1/2})^{m-k}\int_{x_n}^{x_{n+1}}x^{\beta + k}dx
		\cr
		& = \sum_{k = 0}^{m}\binom{m}{k}(-x_{n+1/2})^{m-k}\frac{x_{n+1}^{\beta + k + 1} - x_n^{\beta + k + 1}}{\beta + k + 1}
		\cr
		& = \sum_{k = 0}^{m}\binom{m}{k}(-x_{n+1/2})^{m-k}\frac{(x_{n+1/2} + h/2)^{\beta + k + 1} - (x_{n+1/2} - h/2)^{\beta + k + 1}}{\beta + k + 1}
		\cr
		& = x_{n+1/2}^{\beta + m + 1}\sum_{k = 0}^{m}(-1)^{m-k}\binom{m}{k}\frac{\left(1 + \dfrac{h}{2x_{n+1/2}}\right)^{\beta + k + 1} - \left(1 - \dfrac{h}{2x_{n+1/2}}\right)^{\beta + k + 1}}{\beta + k + 1}
		\cr
		& = 2x_{n+1/2}^{\beta + m + 1}\sum_{l = 0}^\infty\frac{1}{2l+1}\left[\sum_{k = 0}^{m}(-1)^{m-k}\binom{m}{k}\binom{\beta + k}{2l}\right]\left(\frac{h}{2x_{n+1/2}}\right)^{2l+1} = (*)
	\end{align*}
	where we have applied \eqref{e:1xb-1-xb}, and have changed the order of the summation signs. This can be done, because all the involved series are absolutely convergent for the range of parameters considered; indeed, when both $n = 0$ and $a = 0$ do not happen simultaneously, $h / (2x_{n+1/2}) < 1$, and when $n = 0$, $a = 0$, and $\beta > -1$, then $h / (2x_{n+1/2}) = 1$, but $\beta + k + 1 > 0$.
	
	Finally, we simplify the expression in brackets by using \eqref{e:genbincoef} in Lemma \ref{lemma:genbincoef}, and consider only those $l$, such that $2l\le m \Longleftrightarrow l \le \lceil m/2\rceil$:
	\begin{align*}
		(*) & = x_{n+1/2}^{\beta + m + 1}\sum_{l = \lceil m / 2\rceil}^\infty\frac{2}{2l+1}\binom{\beta}{2l - m}\left(\frac{h}{2x_{n+1/2}}\right)^{2l+1}
		\cr
		& = x_{n+1/2}^{\beta + m + 1}\sum_{l = 0}^\infty\frac{2}{2l+2\lceil m / 2\rceil + 1}\binom{\beta}{2l + 2\lceil m / 2\rceil - m}\left(\frac{h}{2x_{n+1/2}}\right)^{2l + 2\lceil m / 2\rceil + 1}
		\cr
		& = \sum_{l = 0}^\infty\frac{h^{2l + 2\lceil m / 2\rceil + 1}}{2^{2l + 2\lceil m / 2\rceil}(2l + 2\lceil m / 2\rceil + 1)}\binom{\beta}{2l + \bmod(m, 2)}x_{n+1/2}^{\beta -2l - \bmod(m, 2)},
	\end{align*}
	where $\bmod(m, 2)$ denotes the remainder of the division of $m$ by $2$.  Considering separately the cases with $m$ even and with $m$ odd, \eqref{e:intann1xbfx} follows. \qed
\end{proof}

\begin{corollary}
	Consider the same parameters as in Lemma \ref{lemma:intann1xbfx}, then
	\begin{equation}
		\label{e:boundKmhb}
		\left|\displaystyle{\int_{x_n}^{x_{n+1}}x^\beta(x - x_{n+1/2})^mdx}\right| \le
		\left\{
		\begin{aligned}
			& K_mh^{m + 1}x_{n+1/2}^\beta, & & m\equiv0\bmod2,
			\cr
			& K_mh^{m + 2}x_{n+1/2}^{\beta - 1}, & & m\equiv1\bmod2.
		\end{aligned}
		\right.
	\end{equation}
\end{corollary}

\begin{proof}
	\begin{align*}
		& \frac{\left|\displaystyle{\int_{x_n}^{x_{n+1}}x^\beta(x - x_{n+1/2})^mdx}\right|}{\displaystyle{h^{2\lceil m / 2\rceil + 1}}x_{n+1/2}^{\beta - \bmod(m, 2)}}
		\cr
		& \quad = \left|\sum_{l = 0}^\infty\frac{1}{2^{2\lceil m / 2\rceil}(2l + 2\lceil m / 2\rceil + 1)}\binom{\beta}{2l + \bmod(m, 2)}\left(\frac{h/2}{a + h(n + 1/2)}\right)^{2l}\right|
		\cr
		& \quad \le \sum_{l = 0}^\infty\frac{1}{2^{2\lceil m / 2\rceil}(2l + 2\lceil m / 2\rceil + 1)}\left|\binom{\beta}{2l + \bmod(m, 2)}\right|\left(\frac{1}{1 + 2a/h}\right)^{2l} \le K_m,
	\end{align*}
	where we have used that the series is convergent, because $(1 / (1 + 2a/h)) \le 1$ (and in the case $(1 / (1 + 2a/h)) = 1$, which happens only when $a = 0$, we are also assuming that $\beta > -1$). Considering separately the cases with $m$ even and with $m$ odd, \eqref{e:boundKmhb} follows. \qed
\end{proof}

\begin{remark} 	The fact that \eqref{e:boundKmhb} behaves like $\mathcal O(h^{m+1})$, when $m$ is even, whereas it behaves like $\mathcal O(h^{m+2})$, when $m$ is odd, happens because there is a change of sign in $(x - x_{n+1})^m$, when $m$ is odd, but not when $m$ is even. It is also important to point out that $K_m$ is independent of $n$, although it obviously depends of $\beta$, $h$ and $a$.
\end{remark}

We also need a result on the behavior of the sums of $\{x_{n+1/2}\}$.
\begin{lemma} \label{lemma:boundmean}
	Let $0 \le a < b$, $\gamma\in\mathbb R$, $N \in\mathbb N$, $h = (b - a) / N$, $x_n = a + hn$, then, there is a constant $K_\gamma\in(0, \infty)$, such that
	\begin{equation}
		\label{e:boundmean}
		\left|h\sum_{n = 0}^{N-1}x_{n+1/2}^\gamma\right| <
		\begin{cases}
			K_\gamma h^{\gamma+1}, & \text{if $a = 0$ and $\gamma < -1$},
			\cr
			K_\gamma - \ln(h/2), & \text{if $a = 0$ and $\gamma = -1$},
			\cr
			K_\gamma, & \text{otherwise}.
		\end{cases} 
	\end{equation}
\end{lemma}

\begin{proof}
	Denote
	$$
	S_N = h\sum_{n = 0}^{N-1}x_{n+1/2}^\gamma.
	$$
	When $\gamma = 0$, the result is trivial. When $a > 0$, the mean of $\{x_{1/2}^\gamma, \ldots, x_{N-1/2}^\gamma\}$ is inside the interval $(a^\gamma, b^\gamma)$, when $\gamma > 0$, or inside $(b^\gamma, a^\gamma)$, when $\gamma < 0$, so $S_N$ is bounded; this argument still holds when $a = 0$ and $\gamma > 0$, because the mean of $\{x_{1/2}^\gamma, \ldots, x_{N-1/2}^\gamma\}$ is then in $(0^\gamma, b^\gamma) = (0, b^\gamma)$, so $S_N$ is bounded, too. On the other hand, when $a = 0$ and $\gamma < 0$:
	\begin{align*}
		0 < S_N & < hx_{1/2}^{\gamma} + \int_h^{b}(x-h/2)^{\gamma}dx
		\cr
		& = \left\{\begin{aligned}
			& 2 + \ln(b - h/2) - \ln(h/2), & & \text{if $\gamma = -1$},
			\\
			& \frac{h^{\gamma+1}}{2^\gamma} + \frac{(b-h/2)^{\gamma+1} - (h/2)^{\gamma+1}}{\gamma + 1}, & & \text{if $\gamma\in(-\infty,0)\backslash\{-1\}$},
		\end{aligned}\right.	
		\cr
		& < \left\{\begin{aligned}
			& 2 + \ln(b) - \ln(h/2), & & \text{if $\gamma = -1$},
			\\
			& \frac{2\gamma + 1}{2^{\gamma + 1}(\gamma + 1)}h^{\gamma+1}, & &\text{if $\gamma\in(-\infty,-1)$},
			\\
			& b^{\gamma+1}\left(\frac{1}{2^\gamma} + \frac{1}{\gamma + 1}\right), & & \text{if $\gamma\in(-1,0)$}.
		\end{aligned}\right.	
	\end{align*} \qed
\end{proof}

\subsection{Proof of  Theorem \ref{theo:intab}}

\begin{proof} When $\beta = 0$, we have the midpoint rule \eqref{e:midpoint}. Therefore, by hypothesis, $f\in\mathcal C^2[a, b]$, so $f''(x)$ is bounded in $[a, b]$, and it follows that
	$$
	\left|\int_{a}^{b} f(x)dx - h\sum_{n = 0}^{N-1}f(x_{n+1/2})\right| \le \frac{K}{N^2}.
	$$
	This is valid for any $a < b$. On the other hand, when $\beta\not=0$, we represent $E$ in \eqref{e:Eint0ab} as a sum of integrals over the intervals $[x_n, x_{n+1}]$:
	\begin{equation}
		\label{e:intabbound}
		E = \sum_{n = 0}^{N-1}\int_{x_n}^{x_{n+1}}x^\beta(f(x) - f(x_{n+1/2}))dx.
	\end{equation}	
	In order to study the behavior of the integral of $x^\beta(f(x) - f(x_{n+1/2}))$ over one single interval $[x_n, x_{n+1}]$, we consider the first terms of the Taylor expansion of $f(x)$ about $x = x_{n+ 1/2}$:
	$$
	f(x) = f(x_{n+1/2}) + f'(x_{n+1/2})(x - x_{n+1/2}) + \frac{f''(\xi_{n,x})}2 (x - x_{n+1/2})^2,
	$$
	for some $\xi_{n,x}$ depending on $x$ (and $n$), such that $x < \xi_{n,x} < x_{n+1/2}$ or $x_{n+1/2} < \xi_{n,x} < x$. Then,
	\begin{align*}
		& \int_{x_n}^{x_{n+1}}x^\beta(f(x) - f(x_{n+1/2}))dx = f'(x_{n+1/2})\int_{x_n}^{x_{n+1}}x^\beta(x - x_{n+1/2})dx
		\cr
		& \quad \qquad + \int_{x_n}^{x_{n+1}}\frac{f''(\xi_{n,x})}2x^\beta(x - x_{n+1/2})^2dx
		\cr
		& \quad = f'(x_{n+1/2})\int_{x_n}^{x_{n+1}}x^\beta(x - x_{n+1/2})dx  + \frac{f''(\xi_{n})}2\int_{x_n}^{x_{n+1}}x^\beta(x - x_{n+1/2})^2dx,
	\end{align*}
	for some $\xi_n\in(x_n, x_{n+1})$. Note that we have applied the generalized mean-value theorem in the second integral, which can be done, because $x^\beta(x - x_{n+1/2})^2$ is positive in $[x_n, x_{n+1}]$. Hence,
	\begin{align}
		\label{e:intaxnxn1bound0}
		\left|\int_{x_n}^{x_{n+1}}x^\beta(f(x) - f(x_{n+1/2}))dx\right| & \le |f'(x_{n+1/2})|\left|\int_{x_n}^{x_{n+1}}x^\beta(x - x_{n+1/2})dx\right|
		\cr
		& \quad + \frac{|f''(\xi_{n})|}2\left|\int_{x_n}^{x_{n+1}}x^\beta(x - x_{n+1/2})^2dx\right|,
	\end{align}
	where the second integral is trivially positive. On the other hand, if $f'(x)$ and $f''(x)$ are continuous in $[a, b]$, they are also bounded, so there are finite positive constants $M_1$ and $M_2$, such that $|f'(x)| \le M_1$, $|f''(x)| \le M_2$, respectively. Therefore,
	\begin{align}
	\label{e:intxnxn1xb}
		& \left|\int_{x_n}^{x_{n+1}}x^\beta(f(x) - f(x_{n+1/2}))dx\right| \le M_1\left|\int_{x_n}^{x_{n+1}}x^\beta(x - x_{n+1/2})dx\right|
		\cr
		& \quad + \frac{M_2}2\left|\int_{x_n}^{x_{n+1}}x^\beta(x - x_{n+1/2})^2dx\right|
		\le h^3\left(M_1K_1x_{n+1/2}^{\beta - 1} + \frac{M_2K_2}2x_{n+1/2}^{\beta}\right),
	\end{align}
	where we have used \eqref{e:boundKmhb} for $m = 1$ and $m = 2$. Now, coming back to \eqref{e:intabbound},
	\begin{align}
		\label{e:totalbound}
		|E| & \le\sum_{n = 0}^{N-1}\left|\int_{x_n}^{x_{n+1}}x^\beta(f(x) - f(x_{n+1/2}))dx\right|
 \le h^3\sum_{n = 0}^{N-1}\left(M_1K_1x_{n+1/2}^{\beta - 1} + \frac{M_2K_2}2x_{n+1/2}^{\beta}\right)
		\cr
		& \le \frac{(b-a)^2}{N^2}\left[M_1K_1\left(h\sum_{n = 0}^{N-1}x_{n+1/2}^{\beta - 1}\right) + \frac{M_2K_2}2\left(h\sum_{n = 0}^{N-1}x_{n+1/2}^{\beta}\right)\right].
	\end{align}
	When $a > 0$, the terms in parentheses are bounded by \eqref{e:boundmean}, which concludes the proof of the first case in \eqref{e:int0ab}. Moreover, when $a = 0$ and $\beta\in(0, 1)$, the same argument applies, and the terms in parentheses are again bounded by \eqref{e:boundmean}, which concludes the proof of the second case in \eqref{e:int0ab}.
	
	In order to prove the third and fourth cases in \eqref{e:int0ab}, i.e., those with $a = 0$ and $\beta\in(-1, 0)$, we proceed in a slightly different way, i.e., we consider the beginning of the Taylor expansion of $f'(x)$ on $x = 0$, i.e., $f'(x_{n+1/2}) = f'(0) + f''(\tilde\xi_n)x_{n+1/2}$, where $\tilde\xi_n\in(0,x_{n+1/2})$, and hence, $|f'(x_{n+1/2})| \le |f'(0)| + |f''(\tilde\xi_n)|x_{n+1/2}$. Introducing this last expression in \eqref{e:intaxnxn1bound0},
	\begin{align*}
		& \left|\int_{x_n}^{x_{n+1}}x^\beta(f(x) - f(x_{n+1/2}))dx\right| \le |f'(0)|\left|\int_{x_n}^{x_{n+1}}x^\beta(x - x_{n+1/2})dx\right|
		\cr
		& \quad + |f''(\tilde\xi_n)|x_{n+1/2}\left|\int_{x_n}^{x_{n+1}}x^\beta(x - x_{n+1/2})dx\right|
 + \frac{|f''(\xi_{n})|}2\left|\int_{x_n}^{x_{n+1}}x^\beta(x - x_{n+1/2})^2dx\right|.
	\end{align*}
	Using \eqref{e:boundKmhb} for $m = 1$ and $m = 2$, and that $|f''(x)|$ is bounded by a constant $M_2$ in $[0, b]$, because it is continuous over that interval, we get
	\begin{equation*}
		\left|\int_{x_n}^{x_{n+1}}x^\beta(f(x) - f(x_{n+1/2}))dx\right| \le |f'(0)| K_1h^3x_{n+1/2}^{\beta - 1} + M_2\frac{2K_1 + K_2}2h^3x_{n+1/2}^{\beta},
	\end{equation*}
	so the equivalent of \eqref{e:totalbound} is
	\begin{align*}
		|E| & \le |f'(0)| K_1h^2\left(h\sum_{n=0}^{N-1}x_{n+1/2}^{\beta - 1} \right) + M_2\left(K_1 + \frac{K_2}2\right)h^2\left(h\sum_{n=0}^{N-1}x_{n+1/2}^{\beta}\right).
	\end{align*}
	By \eqref{e:boundmean}, when $\beta\in(-1, 0)$, the second parenthesized term is bounded by a constant, and the first parenthesized term is bounded by a constant times $h^{\beta-1+1} = h^{\beta}$. Therefore,
	bearing in mind that $h = b / N$, there exist constants $\tilde K_1$ and $\tilde K_2$, such that
	$$
	|E| \le |f'(0)|\frac{\tilde K_1}{N^{\beta+2}} + \frac{\tilde K_2}{N^2}.
	$$
	When $f'(0) = 0$, the first term on the right-hand side of the inequality disappears, whereas, when $f'(0) \not= 0$, everything can be bounded by a constant times $1 / N^{2+\beta}$, because $1 / N^2 \le 1 / N^{2+\beta}$, for $\beta\in(-1,0)$. Therefore, this concludes the proof of the third and fourth cases in \eqref{e:int0ab}. \qed
\end{proof}

\begin{remark} \label{r:abph} Observe that the behavior of $E$ in \eqref{e:Eint0ab} is only dictated by what happens at $x = a$, and indeed, the other end of the integral does not even need to be a constant, provided that the total length of the integral is a natural multiple of $h$. For instance, if $p\in\mathbb N$, then, by \eqref{e:intxnxn1xb}, the following redefinition of $E$ satisfies \eqref{e:int0ab}, too:
\begin{equation*}
E = \int_{a}^{b\pm ph} x^\beta f(x)dx - \sum_{n = 0}^{N\pm p-1}\frac{x_{n+1}^{\beta+1} - x_n^{\beta+1}}{\beta + 1}f(x_{n+1/2}),
\end{equation*}
where, if the negative sign in $\pm$ is considered, $N - p - 1 \ge 0$ is obviously needed. The proof is straightforward, after replacing the upper limit $N-1$ of the summation signs in \eqref{e:totalbound} by $N\pm p - 1$.
\end{remark}

\section{A second-order approximation of the fractional Laplacian}

\label{s:secondorder}

Our aim is to apply Theorem \ref{theo:intab} to the case with $\alpha\not=1$ in \eqref{e:fraclap0pi2}:
\begin{equation}
	\label{e:fraclap0pi2not1}
	(-\Delta)^{\alpha/2}u(s) =  \frac{c_{\alpha}|\sin(s)|^{\alpha-1}}{L^\alpha\alpha(1-\alpha)} \int_0^\pi \frac{\sin^{\alpha}(\eta)(\sin(\eta)u_{ss}(\eta) + 2\cos(\eta)u_{s}(\eta))}{|\sin(s-\eta)|^{\alpha-1}}d\eta.
\end{equation}
However, the integrand in \eqref{e:fraclap0pi2not1} has three singularities, namely at $\eta = 0$, $\eta = \pi$ and $\eta = s$. Therefore, we need to generalize first Theorem \ref{theo:intab} to integrals of the form \eqref{e:Is0pi}. Indeed, we will divide such integrals into two, one on $[0,\pi/2]$ and one on $[\pi/2,\pi]$, and consider different cases, depending on whether the third singularity lies on one or the other interval (the special case $s = \pi/2$ will be treated separately). Each such integral can be written in the form 
\begin{equation}\label{double:sing}
\int_a^{b} x^\beta |x-y|^\gamma f(x,y)dx,
\end{equation}
as follows:
\begin{align*}
I(s) & = \int_0^{\pi/2} \eta^\beta |\eta - s|^\gamma\left(\frac{\sin(\eta)}{\eta}\right)^\beta\left(\frac{\sin(\eta - s)}{\eta - s}\right)^\gamma f(\eta)d\eta
	\cr
& \quad + \int_{\pi/2}^\pi (\pi - \eta)^\beta |\eta - s|^\gamma\left(\frac{\sin(\eta)}{\pi - \eta}\right)^\beta\left(\frac{\sin(\eta - s)}{\eta - s}\right)^\gamma f(\eta)d\eta,
\end{align*}
and then we can apply the modified midpoint rule, by evaluating the regular factor (which varies depending on where $s$ lies) at the midpoints,
and integrating exactly the singular part.

In order to generalize Theorem \ref{theo:intab} to integrals of the form \eqref{double:sing}, the following auxiliary result is required.
\begin{lemma} Let $\gamma\not=-1$, $y\not\in[x_n, x_{n+1}]$, and define
	$$
	E = \frac{\sgn(x_{n+1}-y)|x_{n+1}-y|^{\gamma + 1} - \sgn(x_n-y)|x_n-y|^{\gamma + 1}}{\gamma + 1} - h|x_{n+1/2}-y|^\gamma;
	$$
	then,
	\begin{equation}
		\label{e:quadraturecomparison}
		|E| \le K\,h^3|x_{n+1/2}-y|^{\gamma-2}.
	\end{equation}
	
\end{lemma}

\begin{proof} When $y\not\in[x_n, x_{n+1}]$, $\sgn(x_{n+1}-y)\sgn(x_n-y) > 0$, so
	\begin{align*}
		E & = \frac{(|x_{n+1/2} - y| + h/2)^{\gamma + 1} - (|x_{n+1/2} - y| - h/2)^{\gamma + 1}}{\gamma + 1} - h|x_{n+1/2}-y|^\gamma
		\cr
		& \quad = |x_{n+1/2} - y|^{\gamma + 1}\frac{\left(1 + \dfrac{h/2}{|x_{n+1/2} - y|}\right)^{\gamma + 1} - \left(1 - \dfrac{h/2}{|x_{n+1/2} - y|}\right)^{\gamma + 1}}{\gamma + 1}
 - h|x_{n+1/2}-y|^\gamma
		\cr
		& \quad = h|x_{n+1/2} - y|^{\gamma}\left[\sum_{n = 0}^{\infty}\frac{1}{2n+1}\binom{\gamma}{2n}\left(\dfrac{h/2}{|x_{n+1/2} - y|}\right)^{2n} - 1\right]
		\cr
		& \quad = h^3|x_{n+1/2} - y|^{\gamma-2}\sum_{n = 0}^{\infty}\frac{1}{8n+12}\binom{\gamma}{2n+2}\left(\dfrac{h/2}{|x_{n+1/2} - y|}\right)^{2n},
	\end{align*}
	where we have used \eqref{e:1xb-1-xb}. Bearing in mind that $(h/2)/|x_{n+1/2} - y| \le \tilde K < 1$, 
	\begin{equation*}
		|E|  \le h^3|x_{n+1/2}-y|^{\gamma-2}\sum_{n = 0}^{\infty}\frac{\tilde K^{2n}}{(8n+12)}\left|\binom{\gamma}{2n+2}\right| \le K\,h^3|x_{n+1/2}-y|^{\gamma-2},
	\end{equation*}
	where we have used that the last sum is convergent. \qed
\end{proof}

Now, we prove the following generalization of Theorem \ref{theo:intab}.
\begin{lemma} \label{lemma:E1} Let $f\in\mathcal C^2[a, b]$, with $0 \le a < b$, and let $\beta > -1$, $\gamma\not=-1$, $y\not\in[a,b]$, $N \in\mathbb N$, $h = (b-a) / N$, $x_n = a + hn$, etc. Then, the following expression has the same behavior as \eqref{e:Eint0ab}-\eqref{e:int0ab} in Theorem \ref{theo:intab}:
	\begin{align}
		\label{e:Eintabxbxygf}
		E = & \int_a^{b} x^\beta |x-y|^\gamma f(x)dx - \frac{1}{h}\sum_{n = 0}^{N-1}\frac{x_{n+1}^{\beta+1} - x_n^{\beta+1}}{\beta + 1}
			\cr
& \qquad \cdot \frac{\sgn(x_{n+1} - y)|x_{n+1} - y|^{\gamma + 1} - \sgn(x_n - y)|x_n - y|^{\gamma + 1}}{\gamma + 1}f(x_{n+1/2}).
	\end{align}
\end{lemma}

\begin{proof} Let us rewrite \eqref{e:Eintabxbxygf} by adding and subtracting the same term:
	\begin{align*}
		E & = \Bigg(\int_{a}^{b} x^\beta |x - y|^\gamma f(x)dx - \sum_{n = 0}^{N-1}\frac{x_{n+1}^{\beta+1} - x_n^{\beta+1}}{\beta + 1}|x_{n+1/2} - y|^\gamma f(x_{n+1/2})\Bigg)
		\cr
		& \quad + \frac1h\sum_{n = 0}^{N-1}\frac{x_{n+1}^{\beta+1} - x_n^{\beta+1}}{\beta + 1}\bigg(h|x_{n+1/2}-y|^\gamma 
		\cr
		& \quad \qquad - \frac{\sgn(x_{n+1} - y)|x_{n+1} - y|^{\gamma + 1} - \sgn(x_n - y)|x_n - y|^{\gamma + 1}}{\gamma + 1}\bigg)f(x_{n+1/2})
		\cr
		& = E_1 + E_2.
	\end{align*}
	The behavior of $|E_1|$ is dictated directly by \eqref{e:Eint0ab}-\eqref{e:int0ab} applied to $\tilde f(x) = |x - y|^\gamma f(x)$, which has the same regularity as $f(x)$ on $[a, b]$. On the other hand, in order to bound $|E_2|$, we apply  \eqref{e:quadraturecomparison},  and bear in mind the fact that $|x - y|^{\gamma+2}f(x)$ is continuous on $[a, b]$, and hence, bounded, so there is a constant $M$, such that $|x_{n+1/2}-y|^{\gamma-2}|f(x_{n+1/2})| \le M$, for all $n$:
	\begin{align*}
		|E_2| & \le Kh^2\sum_{n = 0}^{N-1}\frac{x_{n+1}^{\beta+1} - x_n^{\beta+1}}{\beta + 1}|x_{n+1/2}-y|^{\gamma-2}|f(x_{n+1/2})| 
\le K\,M\,h^2\sum_{n = 0}^{N-1}\frac{x_{n+1}^{\beta+1} - x_n^{\beta+1}}{\beta + 1}
		\cr
		& \le \tilde K h^2\left|\sum_{n = 0}^{N-1}\frac{x_{n+1}^{\beta+1} - x_n^{\beta+1}}{\beta + 1} - \int_a^bx^\beta dx\right| + \tilde K h^2\int_a^bx^\beta dx \le \tilde{\tilde K}h^2,
	\end{align*}
	where, in the first term of the last line, we have used \eqref{e:int0ab} applied to $f(x) = 1$. Since the behavior of $|E_1|$ is more restrictive than that of $|E_2|$, \eqref{e:Eintabxbxygf} behaves like $|E_1|$, and hence, like \eqref{e:Eint0ab}-\eqref{e:int0ab}. \qed
\end{proof}

At this point, it is important to remark that, when $a + b = 0$, the behavior of the error \eqref{e:Eintabxbxygf} in Lemma~\ref{lemma:E1} is even better, as shown in the following result.
\begin{lemma}\label{lemma:E2}
Let $f\in\mathcal C^2[-a, a]$, with $a > 0$, and let $\beta > -1$, $\gamma\not=-1$, $y\not\in[-a, a]$, $N \in\mathbb N$, $h = a / N$, $x_n = hn$, etc. Define
\begin{align}
\label{e:Eintaaxbxygf}
		E & = \int_{-a}^{a} |x|^\beta |x-y|^\gamma f(x)dx - \frac{1}{h}\sum_{n = -N}^{N-1}\frac{\sgn(x_{n+1})|x_{n+1}|^{\beta+1} - \sgn(x_n)|x_n|^{\beta+1}}{\beta + 1}
		\cr
& \qquad \cdot \frac{\sgn(x_{n+1} - y)|x_{n+1} - y|^{\gamma + 1} - \sgn(x_n - y)|x_n - y|^{\gamma + 1}}{\gamma + 1} f(x_{n+1/2}).
\end{align}
Then, there exists $K > 0$, such that
$$
|E| \le \frac{K}{N^2}.
$$
\end{lemma}

\begin{proof} The proof is very similar to that of Lemma \ref{lemma:E1}. Let us rewrite \eqref{e:Eintaaxbxygf} by adding and subtracting the same term:
\begin{align*}
E & = \Bigg(\int_{-a}^{a} |x|^\beta |x-y|^\gamma f(x)dx
	\cr
& \quad - \sum_{n = -N}^{N-1}\frac{\sgn(x_{n+1})|x_{n+1}|^{\beta+1} - \sgn(x_n)|x_n|^{\beta+1}}{\beta + 1}|x_{n+1/2} - y|^\gamma f(x_{n+1/2})\Bigg)
	\cr
& \quad
+ \frac{1}{h}\sum_{n = -N}^{N-1}\frac{\sgn(x_{n+1})|x_{n+1}|^{\beta+1} - \sgn(x_n)|x_n|^{\beta+1}}{\beta + 1}\bigg[h|x_{n+1/2} - y|^\gamma 
	\cr
& \quad \qquad
- \frac{\sgn(x_{n+1} - y)|x_{n+1} - y|^{\gamma + 1} - \sgn(x_{n} - y)|x_{n} - y|^{\gamma + 1}}{\gamma + 1}\bigg] f(x_{n+1/2})
	\cr
& = E_1 + E_2.
\end{align*}
In order to bound $|E_1|$, we define $g(x) = |x-y|^\gamma f(x) + |x+y|^\gamma f(-x)$. Then, it is immediate to check that
	\begin{align*}
		E_1 & = \int_0^{a} x^\beta g(x)dx - \sum_{n = 0}^{N-1}\frac{x_{n+1}^{\beta+1} - x_n^{\beta+1}}{\beta + 1}g(x_{n+1/2}).
	\end{align*}
	Furthermore,
	\begin{align*}
		g'(x) & = \gamma \sgn(x - y)|x - y|^{\gamma-1}f(x) + |x - y|^\gamma f'(x)
		\cr
		& \qquad + \gamma \sgn(x + y)|x + y|^{\gamma-1}f(-x) - |x + y|^\gamma f'(-x) \Longrightarrow g'(0) = 0,
	\end{align*}
	so, from \eqref{e:int0ab}, $|E_1| \le K / N^2$.

In order to bound $|E_2|$, we apply  again \eqref{e:quadraturecomparison}, and that there is a constant $M$, such that $|x_{n+1/2}-y|^{\gamma-2}|f(x_{n+1/2})| \le M$, for all $n$:
\begin{align*}
|E_2| & \le Kh^2\sum_{n = -N}^{N-1}\frac{\sgn(x_{n+1})|x_{n+1}|^{\beta+1} - \sgn(x_n)|x_n|^{\beta+1}}{\beta + 1}|x_{n+1/2}-y|^{\gamma-2} |f(x_{n+1/2})|
	\cr
& \le KMh^2\sum_{n = -N}^{N-1}\frac{\sgn(x_{n+1})|x_{n+1}|^{\beta+1} - \sgn(x_n)|x_n|^{\beta+1}}{\beta + 1}
	\cr
& = KMh^2\sum_{n = -N}^{N-1}\int_{x_n}^{x_{n+1}}|x|^\beta dx = K\,M\,h^2\int_{-a}^{a}|x|^\beta  dx = \tilde Kh^2. 
\end{align*}
Therefore, both $|E_1|$ and $|E_2|$ behave like $\mathcal O(h^2) = \mathcal O(1/ N^2)$. \qed
\end{proof}

Lemmas \ref{lemma:E1} and \ref{lemma:E2} allow us to construct a second-order numerical approximation of \eqref{e:fraclap0pi2not1}; in fact, we will prove a more general result. Taking into account \eqref{e:xxis}, we need to adjust our notation; in the rest of the paper, $s$ will be used when we work on the $[0,\pi]$ interval, and $x$, when we work on the whole real line $\mathbb R$. Therefore, given $N\in\mathbb N$, $r\in\mathbb N$, we define the following family of equally spaced nodes that divide $[0,\pi]$ in $2rN$ equal-length subintervals:
\begin{equation}
\label{e:snr}
\tilde s_{n} = \frac{\pi n}{2rN} = h_rn, \text{ where } h_r \equiv \frac{\pi}{2rN}, \ n\in\{0, \ldots, 2rN\};
\end{equation}
hence, $\tilde s_{n+1/2} = h_r(n+1/2)$, $\tilde s_{n+1} = h_r(n+1)$, and
$$
[0, \pi] = \bigcup_{n = 0}^{2Nr - 1}[\tilde s_n, \tilde s_{n+1}].
$$
On the other hand, \eqref{e:fraclap0pi2not1} will be approximated at the nodes $s_j$, defined as
\begin{equation}
\label{e:sj}
s_j \equiv \frac{(2j+1)\pi}{2N}, \quad j\in\{0, \ldots, N-1\}.
\end{equation}
Observe that, in \eqref{e:sj}, we have preferred to write $s_j$ instead of $s_{j+1/2}$, because it is less cumbersome, and there is absolutely no risk of confusion with the nodes $\tilde s_n$.

A crucial property of the nodes $s_j$, because it enables applying correctly the quadrature formulas considered in this paper, is that, for any fixed $r\in\mathbb N$, there are two intervals of the form $[\tilde s_n, \tilde s_{n+1}]$, such that $s_j$ is respectively, their right and left end:
\begin{equation}
\label{e:sjsnr}
s_j = \tilde s_n \Longrightarrow \frac{(2j+1)\pi}{2N} = \frac{\pi n}{2rN} \Longrightarrow n = (2j + 1)r;
\end{equation}
i.e., those intervals are precisely $[\tilde s_{(2j+1)r-1}, s_j]$ and $[s_j, \tilde s_{(2j+1)r+1}]$.

Gathering all the previous arguments, we have the following theorem, which is the main result of this paper, and which allows us to approximate numerically \eqref{e:Is0pi}, from which \eqref{e:fraclap0pi2not1} is a particular case times a factor outside the integral sign.
\begin{theorem} \label{theo:int0pibg} Let $f\in\mathcal C^2[0, \pi]$, $\beta > 0$, $\gamma>-1$, $N\in\mathbb N$, $r\in\mathbb N$, $h_r = \pi / (2rN)$, $s_j$ as defined in \eqref{e:sj}, and $\tilde s_n$ as defined in \eqref{e:snr}.
For a given $s_j$, let $I(s_j)$ denote \eqref{e:Is0pi} evaluated at $s = s_j$:
\begin{equation}
\label{e:Isj0pi}
I(s_j) = \int_0^\pi \sin^\beta(\eta)|\sin(\eta - s_j)|^\gamma f(\eta)d\eta.
\end{equation}
Define
\begin{align}
\label{e:A1j}
A_{1,j} & = \frac1{h_r}\sum_{n = 0}^{rN-1}\left(\frac{\sin(\tilde s_{n+1/2})}{\tilde s_{n+1/2}}\right)^\beta\left(\frac{\sin(\tilde s_{n+1/2}-s_j)}{\tilde s_{n+1/2}- s_j}\right)^\gamma\frac{\tilde s_{n+1}^{\beta+1} - \tilde s_n^{\beta + 1}}{\beta + 1}
\cr
& \quad
\cdot\frac{\sgn(\tilde s_{n+1} - s_j)|\tilde s_{n+1} - s_j|^{\gamma + 1} - \sgn(\tilde s_{n} - s_j)|\tilde s_{n} - s_j|^{\gamma + 1}}{\gamma + 1}f(\tilde s_{n+1/2}),
	\\
\label{e:A2j}
A_{2,j} & = \frac1{h_r}\sum_{n = rN}^{2rN-1}\left(\frac{\sin(\tilde s_{n+1/2})}{\pi - \tilde s_{n+1/2}}\right)^\beta\left(\frac{\sin(\tilde s_{n+1/2}-s_j)}{\tilde s_{n+1/2}- s_j}\right)^\gamma
\frac{(\pi-\tilde s_n)^{\beta+1} - (\pi-\tilde s_{n+1})^{\beta + 1}}{\beta + 1}
\cr
& \quad
\cdot\frac{\sgn(\tilde s_{n+1} - s_j)|\tilde s_{n+1} - s_j|^{\gamma + 1} - \sgn(\tilde s_{n} - s_j)|\tilde s_{n} - s_j|^{\gamma + 1}}{\gamma + 1}f(\tilde s_{n+1/2}).
\end{align}
Then, $A_{1,j} + A_{2,j}$ approximate numerically $I(s_j)$, and, there is $K > 0$, such that, for all  $j\in\{0, \ldots, N-1\}$, the numerical error
\begin{equation}
	\label{e:EIsj}
	E_j = I(s_j) - A_{1,j} - A_{2,j}
\end{equation}
satisfies
\begin{equation*}
|E_j| \le \frac{K}{r^2}.
\end{equation*}

\end{theorem}

\begin{proof}
Let us rewrite \eqref{e:Isj0pi}:
\begin{align}
\label{e:Isj}
I(s_j) & = \int_0^{\pi/2} \eta^\beta |\eta - s_j|^\gamma\left(\frac{\sin(\eta)}{\eta}\right)^\beta\left(\frac{\sin(\eta - s_j)}{\eta - s_j}\right)^\gamma f(\eta)d\eta
	\cr
& \quad + \int_{\pi/2}^\pi (\pi - \eta)^\beta |\eta - s_j|^\gamma\left(\frac{\sin(\eta)}{\pi - \eta}\right)^\beta\left(\frac{\sin(\eta - s_j)}{\eta - s_j}\right)^\gamma f(\eta)d\eta
	\cr
& = I_{1,j} + I_{2,j};
\end{align}
note that $\eta - s_j\in(-\pi, \pi)$, which implies that $\sin(\eta - s_j) / (\eta - s_j) > 0$ (as usual, $\sin(0)/0\equiv1$), and hence, it is possible to replace $|\cdot|^\gamma$ by $(\cdot)^\gamma$. On the other hand, the structure of $I_1$ and $I_2$ in \eqref{e:Isj} suggests defining
\begin{align*}
g_{1,j}(s) & = \left(\frac{\sin(s)}{s}\right)^\beta\left(\frac{\sin(s - s_j)}{s - s_j}\right)^\gamma f(s),
	\cr
g_{2,j}(s) & = \left(\frac{\sin(s)}{\pi - s}\right)^\beta\left(\frac{\sin(s - s_j)}{s - s_j}\right)^\gamma f(s),
\end{align*}
which, for any $j$, have the same regularity as $f(s)$ over $s\in(0,\pi)$. Therefore, $E_j$, as defined in \eqref{e:EIsj}, becomes
\begin{align*}
E_j & = \Bigg(\int_0^{\pi/2} \eta^\beta|\eta - s_j|^\gamma g_{1,j}(\eta)d\eta - \frac1{h_r}\sum_{n = 0}^{rN-1}\frac{\tilde s_{n+1}^{\beta+1} - \tilde s_n^{\beta + 1}}{\beta + 1}
\cr
& \qquad
\cdot\frac{\sgn(\tilde s_{n+1} - s_j)|\tilde s_{n+1} - s_j|^{\gamma + 1} - \sgn(\tilde s_{n} - s_j)|\tilde s_{n} - s_j|^{\gamma + 1}}{\gamma + 1}g_{1,j}(s_{n+1/2})\Bigg)
\cr
& \quad + \Bigg(\int_{\pi/2}^\pi (\pi-\eta)^\beta|\eta - s_j|^\gamma g_{2,j}(\eta)d\eta - \frac1{h_r}\sum_{n = rN}^{2rN-1}\frac{(\pi - \tilde s_n)^{\beta+1} - (\pi - \tilde s_{n+1})^{\beta + 1}}{\beta + 1}
\cr
& \qquad \cdot\frac{\sgn(\tilde s_{n+1} - s_j)|\tilde s_{n+1} - s_j|^{\gamma + 1} - \sgn(\tilde s_{n} - s_j)|\tilde s_{n} - s_j|^{\gamma + 1}}{\gamma + 1}g_{2,j}(s_{n+1/2})\Bigg)
\cr
& = E_{1,j} + E_{2,j}.
\end{align*}
Let us assume without loss of generality that $r \ge 2$; indeed, it is enough to adjust later the constant $K$ in \eqref{e:EIsj} for $r = 1$, if needed. In this way, we make sure that there is at most one singularity in each interval $[\tilde s_n, \tilde s_{n+1}]$. Suppose that $s_j\in(0, \pi / 2)$. Then, in the integral in $E_{2,j}$, $\eta - s_j > 0$, so there is only one singularity in that integral: that at $\eta = \pi$, i.e., $(\pi - \eta)^\beta$, with $\beta > 0$. Furthermore, after applying the change of variable $\eta \longleftrightarrow \pi - \eta$,
\begin{align*}
E_{2,j} & = \int_0^{\pi/2}\eta^\beta|\pi - \eta - s_j|^\gamma g_{2,j}(\pi - \eta)d\eta - \frac1{h_r}\sum_{n = 0}^{rN-1}g_{2,j}(\pi - s_{n+1/2})\frac{\tilde s_{n+1}^{\beta+1} - \tilde s_n^{\beta + 1}}{\beta + 1}
\cr
& \quad \cdot\frac{\sgn(\pi - \tilde s_n - s_j)|\pi - \tilde s_n - s_j|^{\gamma + 1} - \sgn(\pi - \tilde s_{n+1} - s_j)|\pi - \tilde s_{n+1} - s_j|^{\gamma + 1}}{\gamma + 1},
\end{align*}
so, by Lemma \ref{lemma:E1}, since $\beta > 0$, there  exists $\tilde K > 0$, such that $|E_{2,j}| \le \tilde K / (rN) = K / r^2$. On the other hand, in $E_{1,j}$, we decompose provisionally the integration domain in three pieces: $[0, \pi/2] = [0, s_j - \pi / (4N)] \cup [s_j - \pi / (4N), s_j + \pi / (4N)] \cup [s_j + \pi / (4N), \pi/2]$.
Then, reasoning as in \eqref{e:sjsnr},
\begin{equation*}
\tilde s_n = s_j \pm \frac{\pi}{4N} \Longrightarrow \frac{\pi n}{2rN} = \frac{(2j+1)\pi}{2N} \pm \frac{\pi}{4N}  \Longrightarrow n = (2j + 1)r \pm \frac{r}{2}.
\end{equation*}
However, $n$ must be an integer, so we take for instance the floor function of $r / 2$. Denoting $n_1$ and $n_2$ the values corresponding respectively to the plus and minus signs:
\begin{equation}
\label{e:n1n2}
\begin{split}
n_1 & = (2j + 1)r - \left\lfloor\frac{r}{2}\right\rfloor \Longrightarrow \tilde s_{n_1} = \frac{\pi}{2rN}\left((2j + 1)r - \left\lfloor\frac{r}{2}\right\rfloor\right)  = s_j - h_r\left\lfloor\frac{r}{2}\right\rfloor,
	\cr
n_2 & = (2j + 1)r + \left\lfloor\frac{r}{2}\right\rfloor \Longrightarrow \tilde s_{n_2} = \frac{\pi}{2rN}\left((2j + 1)r + \left\lfloor\frac{r}{2}\right\rfloor\right) =  s_j + h_r\left\lfloor\frac{r}{2}\right\rfloor,
\end{split}
\end{equation}
so the actual decomposition of the integration domain is $[0,\pi/2] = [0, \tilde s_{n_1}] \cup [\tilde s_{n_1}, \tilde s_{n_2}] \cup [\tilde s_{n_2}, \pi/2]$; note also that $n_1\in\mathcal O(r)$ and $n_2\in\mathcal O(r)$. Now, let us decompose $E_{1,j}$ in three pieces, corresponding to those three intervals:
\begin{align*}
E_{1,j} & = \Bigg(\int_0^{\tilde s_{n_1}} \eta^\beta|\eta - s_j|^\gamma g_{1,j}(\eta)d\eta - \frac1{h_r}\sum_{n = 0}^{n_1 - 1}g_{1,j}(\tilde s_{n+1/2})\frac{\tilde s_{n+1}^{\beta+1} - \tilde s_n^{\beta + 1}}{\beta + 1}
\cr
& \qquad\cdot\frac{\sgn(\tilde s_{n+1} - s_j)|\tilde s_{n+1} - s_j|^{\gamma + 1} - \sgn(\tilde s_{n} - s_j)|\tilde s_{n} - s_j|^{\gamma + 1}}{\gamma + 1}\Bigg)
	\cr
& \quad + \Bigg(\int_{s_{n_1}}^{s_{n_2}} \eta^\beta|\eta - s_j|^\gamma g_{1,j}(\eta)d\eta - \frac1{h_r}\sum_{n = n_1}^{n_2-1}g_{1,j}(\tilde s_{n+1/2})\frac{\tilde s_{n+1}^{\beta+1} - \tilde s_n^{\beta + 1}}{\beta + 1}
\cr
& \qquad
\cdot\frac{\sgn(\tilde s_{n+1} - s_j)|\tilde s_{n+1} - s_j|^{\gamma + 1} - \sgn(\tilde s_{n} - s_j)|\tilde s_{n} - s_j|^{\gamma + 1}}{\gamma + 1}\Bigg)
	\cr
& \quad + \Bigg(\int_{s_{n_2}}^{\pi/2} \eta^\beta|\eta - s_j|^\gamma g_{1,j}(\eta)d\eta - \frac1{h_r}\sum_{n = n_2}^{rN-1}g_{1,j}(\tilde s_{n+1/2})\frac{\tilde s_{n+1}^{\beta+1} - \tilde s_n^{\beta + 1}}{\beta + 1}
\cr
& \qquad
\cdot\frac{\sgn(\tilde s_{n+1} - s_j)|\tilde s_{n+1} - s_j|^{\gamma + 1} - \sgn(\tilde s_{n} - s_j)|\tilde s_{n} - s_j|^{\gamma + 1}}{\gamma + 1}\Bigg)
	\cr
& = E_{1,j}^a + E_{1,j}^b + E_{1, j}^c.
\end{align*}
In $E_{1,j}^a$, the only singularity in its integral is located at $\eta = 0$, i.e., $|\eta|^\beta$, with $\beta > 0$,  and in $E_{1,j}^c$, the integral has no singularity, so by Lemma \ref{lemma:E1} and Remark \ref{r:abph}, there exists $K > 0$, such that $|E_{1,j}^a| \le K / r^2$ and $|E_{1,j}^c| \le K / r^2$. With respect to $E_{1,j}^b$, we perform the change of variable $\eta \longleftrightarrow \eta + s_j$:
\begin{align*}
E_{1,j}^b & =  \int_{-h_r\lfloor r/2\rfloor}^{h_r\lfloor r/2\rfloor} (\eta+s_j)^\beta|\eta|^\gamma g_{1,j}(\eta + s_j)d\eta - \frac1{h_r}\sum_{n = -\lfloor r/2\rfloor}^{\lfloor r/2\rfloor-1}g_{1,j}(\tilde s_{n+1/2} + s_j)
	\cr
& \quad\cdot
\frac{(\tilde s_{n+1} + s_j)^{\beta+1} - (\tilde s_n + s_j)^{\beta + 1}}{\beta + 1}\frac{\sgn(\tilde s_{n+1})|\tilde s_{n+1}|^{\gamma + 1} - \sgn(\tilde s_{n})|\tilde s_{n}|^{\gamma + 1}}{\gamma + 1}.
\end{align*}
Then, since $\eta + s_j \not\in[-h_r\lfloor r/2\rfloor, h_r\lfloor r/2\rfloor]$, the only singularity in the integral is located at $\eta = 0$, i.e., $|\eta|^\gamma$, with $\gamma > -1$.  Therefore, again by Lemma \ref{lemma:E1} and Remark \ref{r:abph}, there exists $K > 0$, such that $|E_{2,1}| \le K / r^2$. Since $|E_{1,j}^a|$, $|E_{1,j}^b|$, $|E_{1,j}^c|$ and $|E_{2,j}|$ are all bounded by $K / r^2$, it follows that $|E_j| \le K / r^2$, too.  

Suppose now that $s_j\in(\pi / 2, \pi)$. Then, the proof that $|E_j| \le K / r^2$ is identical to that when $s_j\in(0, \pi / 2)$. $E_{1,j}$ poses no problems, and, to bound $E_{2,j}$, we define $n_1$ and $n_2$ as in \eqref{e:n1n2}, and decompose the integration domain $[\pi/2, \pi]$ (with no change of variable) in three pieces: $[\pi/2, \pi] = [\pi/2, \tilde s_{n_1}] \cup [\tilde s_{n_1}, \tilde s_{n_2}] \cup [\tilde s_{n_2}, \pi]$, and write $E_{2,j} = E_{2, j}^a + E_{2, j}^b + E_{2, j}^c$, where each term is associated to its respective subinterval. The rest of the details are absolutely identical, so we do not repeat them again.

In order to conclude the proof, there is one last case to consider, when $s_j = \pi / 2$, which can only happen when $N$ is odd, and $j = (N -1)/2$ and for which we need to reason in a different way. To deal with that case,  it is straightforward to check that $E_j$ in \eqref{e:EIsj} becomes
\begin{align*}
	E_j & = \int_0^{\pi/2} \sin^\beta(\eta)\sin^\gamma(\pi/2 - \eta ) (f(\eta) + f(\pi - \eta))d\eta
	\cr
	& \quad - 
	\frac1{h_r}\sum_{n = 0}^{rN-1}\left(\frac{\sin(\tilde s_{n+1/2})}{\tilde s_{n+1/2}}\right)^\beta\left(\frac{\sin(\pi/2-\tilde s_{n+1/2})}{\pi/2-\tilde s_{n+1/2}}\right)^\gamma\frac{\tilde s_{n+1}^{\beta+1} - \tilde s_n^{\beta + 1}}{\beta + 1}
	\cr
	& \qquad
	\cdot\frac{(\pi/2 - \tilde s_n)^{\gamma + 1} - (\pi/2-\tilde s_{n+1})^{\gamma + 1}}{\gamma + 1}(f(\tilde s_{n+1/2}) + f(\pi - \tilde s_{n+1/2})).
\end{align*}
This suggests defining
\begin{align*}
g_a(s) & = \left(\frac{\sin(s)}{s}\right)^\beta\left(\frac{\sin(\pi/2 - s)}{\pi/2 - s}\right)^\gamma (f(s) + f(\pi - s)),
	\cr
g_b(s) & = \left(\frac{\sin(s)}{s}\right)^\gamma\cos^\beta(s) (f(\pi/2 - s) + f(\pi/2 + s)),
\end{align*}
which have respectively the same regularity as $f(s) + f(\pi - s)$ and $f(\pi/2 - s) + f(\pi/2 + s)$ on $s\in[0, \pi/2]$.  Let us divide  the integration domain $[0, \pi/2]$ in two subintervals of equal or approximately equal length:
$$
\tilde s_n = \frac\pi4 \Longrightarrow \frac{\pi n}{2rN} = \frac\pi4  \Longrightarrow n = \frac{rN}{2} \Longrightarrow n_a = \left\lceil\frac{rN}{2} \right\rceil\text{ and } \tilde s_{n_a} = h_rn_a\ge\frac\pi4.
$$
Then, denoting $n_b = rN - n_a$, $\tilde s_{n_b} = \pi/2 - \tilde s_{n_a} \le \pi/4$,
\begin{align*}
E_j & = \Bigg(\int_0^{\tilde s_{n_a}} \eta^\beta(\pi/2 - \eta)^\gamma g_a(\eta)d\eta - 
	\frac1{h_r}\sum_{n = 0}^{n_a}\frac{\tilde s_{n+1}^{\beta+1} - \tilde s_n^{\beta + 1}}{\beta + 1}
\cr
& \qquad \cdot	
	\frac{(\pi/2 - \tilde s_n)^{\gamma + 1} - (\pi/2-\tilde s_{n+1})^{\gamma + 1}}{\gamma + 1}g_a(\tilde s_{n+1/2})\Bigg)
 + \Bigg(\int_0^{\tilde s_{n_b}} \eta^\gamma g_b(\eta)d\eta	
	\cr
	& \quad - \sum_{n = 0}^{n_b-1}\frac{\tilde s_{n+1}^{\gamma + 1}
		- \tilde s_{n}^{\gamma + 1}}{\gamma + 1}\frac{(\pi/2-\tilde s_{n})^{\beta+1} - (\pi/2-\tilde s_{n+1})^{\beta+1}}{(\beta + 1)h_r(\pi/2 - \tilde s_{n+1/2})^\beta}
g_b(\tilde s_{n+1/2})\Bigg)
\cr
& = E_a + E_b.
\end{align*}
$E_a$ has only one singularity at the integral, located at $\eta = 0$, namely  $\eta^\beta$, with $\beta > 0$, so by Lemma \ref{lemma:E1} and Remark \ref{r:abph}, there exists $K > 0$, such that $|E_a| \le K / r^2$. On the other hand, in $E_b$,  where we have performed the change of variable $\eta \longleftrightarrow \pi/2 - \eta$, there is only one singularity at the integral, located also at $\eta = 0$, namely  $\eta^\gamma$, with $\gamma > -1$. Let us rewrite $E_b$ by subtracting and adding the same term:
\begin{align*}
E_b & = \left(\int_0^{\tilde s_{n_b}} \eta^\gamma g_b(\eta)d\eta - \sum_{n = 0}^{n_b-1}\frac{\tilde s_{n+1}^{\gamma + 1} - \tilde s_{n}^{\gamma + 1}}{\gamma + 1}g_b(\tilde s_{n+1/2})\right)
\cr
& \quad + \left(\sum_{n = 0}^{n_b-1}\frac{\tilde s_{n+1}^{\gamma + 1}
	- \tilde s_{n}^{\gamma + 1}}{\gamma + 1}\left[1 - \frac{(\pi/2-\tilde s_{n})^{\beta+1} - (\pi/2-\tilde s_{n+1})^{\beta+1}}{(\beta + 1)h_r(\pi/2 - \tilde s_{n+1/2})^\beta}\right]
g_b(\tilde s_{n+1/2})\right)
\cr
& = E_{b, 1} + E_{b, 2}.
\end{align*}
With respect to $E_{b, 1}$, since $g_b'(0) = 0$, by Theorem \ref{theo:intab} and Remark \ref{r:abph}, there exists $K$ such that $|E_{b,1}| \le K/r^2$; and, with respect to $E_{b, 2}$, using \eqref{e:quadraturecomparison}, and that there exists $M$ such that $|g_b(s)| \le M$, for all $s\in[0,\pi/2]$ (which is possible, because $g_b(s)$ is continuous on $[0,\pi/2]$),
\begin{align*}
|E_{b,2}| & \le \sum_{n = 0}^{n_b-1}\frac{\tilde s_{n+1}^{\gamma + 1}
	- \tilde s_{n}^{\gamma + 1}}{\gamma + 1}\frac{|g_b(\tilde s_{n+1/2})|}{h_r(\pi/2 - \tilde s_{n+1/2})^\beta}
	\cr
& \quad \cdot \bigg|h_r(\pi/2 - \tilde s_{n+1/2})^\beta 
- \frac{(\pi/2-\tilde s_{n})^{\beta+1} - (\pi/2-\tilde s_{n+1})^{\beta+1}}{\beta + 1}\bigg|
	\cr
& \le \sum_{n = 0}^{n_b-1}\frac{\tilde s_{n+1}^{\gamma + 1}
	- \tilde s_{n}^{\gamma + 1}}{\gamma + 1}\frac{|g_b(\tilde s_{n+1/2})|}{h_r(\pi/2 - \tilde s_{n+1/2})^\beta}Kh_r^3(\pi/2 - \tilde s_{n+1/2})^{\beta-2}
	\cr
& \le Kh_r^2\sum_{n = 0}^{n_b-1}\frac{\tilde s_{n+1}^{\gamma + 1}
	- \tilde s_{n}^{\gamma + 1}}{\gamma + 1}\frac{|g_b(\tilde s_{n+1/2})|}{(\pi/2 - \tilde s_{n+1/2})^{2}}
 \le \frac{KM}{(\pi/2 - \pi/4)^2}\left(\frac{\pi}{2rN}\right)^2\frac{\tilde s_{n_b}^{\gamma+1}}{\gamma+1} \le \frac{\tilde K}{r^2}.
\end{align*}
Therefore, when $\tilde s_n = \pi/2$, there exists also $K > 0$, such that $|E_j| \le K / r^2$.

Finally, we take a $K$ large enough, such that, for $j \in \{0, \ldots, N-1\}$, $|E_j| \le K / r^2$, which concludes the proof of the theorem. \qed
\end{proof}

Now, we can adapt the main result to the particular case of the fractional Laplacian, which gives the following error estimate of the method.
\begin{corollary}  \label{corol:fraclap} Let $u(x) : \mathbb R \to\mathbb C$, such that, after defining $u(s)\equiv u(L\cot(s))$, with $L > 0$, we have that $\sin(s)u_{ss}(s) + 2\cos(s)u_{s}(s)\in\mathcal C^2[0, \pi]$. Let $\alpha \in(0,1)\cup(1,2)$. $N\in\mathbb N$, $r\in\mathbb N$, $h_r = \pi / (2rN)$, $s_j$ as defined in \eqref{e:sj}, and $\tilde s_n$ as defined in \eqref{e:snr}. Let  $(-\Delta)^{\alpha/2}u(s_j)$, as defined in \eqref{e:fraclap0pi2not1}, denote the fractional Laplacian of $u(x)$ at $x = x_j =  L\cot(s_j)$. Define
	\begin{align*}
		A_{1,j} & = \frac1{h_r} \frac{c_{\alpha}|\sin(s)|^{\alpha-1}}{L^\alpha\alpha(1-\alpha)}\sum_{n = 0}^{rN-1}\left(\frac{\sin(\tilde s_{n+1/2})}{\tilde s_{n+1/2}}\right)^\alpha\left(\frac{\sin(\tilde s_{n+1/2}-s_j)}{\tilde s_{n+1/2}- s_j}\right)^{1-\alpha}
		\cr
		& \quad\cdot\frac{\tilde s_{n+1}^{\alpha+1} - \tilde s_n^{\alpha + 1}}{\alpha + 1}\frac{\sgn(\tilde s_{n+1} - s_j)|\tilde s_{n+1} - s_j|^{2 - \alpha} - \sgn(\tilde s_{n} - s_j)|\tilde s_{n} - s_j|^{2 - \alpha}}{2 - \alpha}
		\cr
		& \quad\cdot(\sin(\tilde s_{n+1/2})u_{ss}(\tilde s_{n+1/2}) + 2\cos(\tilde s_{n+1/2})u_{s}(\tilde s_{n+1/2})),
	\end{align*}
	\begin{align*}
		A_{2,j} & = \frac1{h_r} \frac{c_{\alpha}|\sin(s)|^{\alpha-1}}{L^\alpha\alpha(1-\alpha)}\sum_{n = rN}^{2rN-1}\left(\frac{\sin(\tilde s_{n+1/2})}{\pi - \tilde s_{n+1/2}}\right)^\alpha\left(\frac{\sin(\tilde s_{n+1/2}-s_j)}{\tilde s_{n+1/2}- s_j}\right)^{1-\alpha}
		\cr
		& \quad\cdot\frac{(\pi-\tilde s_n)^{\alpha+1} - (\pi-\tilde s_{n+1})^{\alpha + 1}}{\alpha + 1}
		\cr
		& \quad
		\cdot\frac{\sgn(\tilde s_{n+1} - s_j)|\tilde s_{n+1} - s_j|^{2 - \alpha} - \sgn(\tilde s_{n} - s_j)|\tilde s_{n} - s_j|^{2 - \alpha}}{2 - \alpha}
		\cr
		& \quad\cdot(\sin(\tilde s_{n+1/2})u_{ss}(\tilde s_{n+1/2}) + 2\cos(\tilde s_{n+1/2})u_{s}(\tilde s_{n+1/2})).
	\end{align*}
	Then, there exists $K > 0$, such that, for all $k\in\{0, \ldots, N-1\}$,
	\begin{equation*}
		E_j = (-\Delta)^{\alpha/2}u(s_j) - A_{1,j} - A_{2,j} \Longrightarrow |E_j| \le \frac{K}{r^2}.
	\end{equation*}
\end{corollary}

\begin{proof} It is enough to apply Theorem \ref{theo:int0pibg} to $f(s) = \sin(s)u_{ss}(s) + 2\cos(s)u_{s}(s)$, taking $\beta = \alpha > 0$ and $\gamma = 1 - \alpha > -1$, and, for each $s_j$, multiply the result by $c_{\alpha}|\sin(s_j)|^{\alpha-1} / (L^\alpha\alpha(1-\alpha))$. \qed
\end{proof}

Note that it is possible to compute directly $A_{1,j}$ and $A_{2,j}$ in Theorem~\ref{theo:int0pibg} and Corollary~\ref{corol:fraclap}, with one single caveat: since the computer uses internally a floating system  to represent the numbers, when $s_n - s_j = 0$ or $s_{n+j} - s_j = 0$, it could happen that the computer stores the results not as being exactly zero, but as infinitesimal numbers, so, in those cases, $\sgn(\tilde s_n - s_j)\not = 0$, and $\sgn(\tilde s_{n+1} - s_j)\not = 0$, which could have disastrous effects in the accuracy of the results. However, this problem is easily avoided, by doing the following:
\begin{equation}
\label{e:sgnsnsj}
\begin{split}
\tilde s_n - s_j =  h_r(n - (2j+1)r) & \Longrightarrow \sgn(\tilde s_n - s_j) \equiv \sgn(n - (2j+1)r),
\cr
\tilde s_{n+1} - s_j =  h_r(n+1 - (2j+1)r) & \Longrightarrow \sgn(\tilde s_{n+1} - s_j) \equiv \sgn(n + 1 - (2j+1)r),
\end{split}
\end{equation}
so the cases $s_n - s_j = 0$ or $s_{n+j} - s_j = 0$ are accurately detected.

On the other hand, we must compute $N$ occurrences of both $A_{1,j}$ and $A_{2,j}$, which requires a number of operations of the order of $\mathcal O(rN^2)$, and becomes computationally expensive, even for not too large values of $N$. In order to minimize this problem, we explain in the following section how to reduce the computational cost to just $\mathcal O(rN\log(N))$ operations, which will allow applying the numerical method to very large values of $N$.

\section{Numerical convolutions and the approximation of singular integrals}

\label{s:numericalconvo}

\subsection{The fast convolution technique}

Let us consider two $N$-periodic sequences of complex numbers $u = \{u_m\}$ and $v = \{v_m\}$, with $m\in\mathbb Z$. Then, the convolution of $u$ and $v$, denoted by $u\ast v$, is a new sequence defined as
\begin{equation}
\label{e:defconf}
(u\ast v)_{m} \equiv \sum\limits_{n=0}^{N-1}u_{n}v_{m-n},\quad m\in\mathbb Z.
\end{equation}
It is straightforward to check that $u\ast v$ is also $N$-periodic, and that $u\ast v = v\ast u$. Moreover, given an $N$-periodic sequence $u$, its discrete Fourier transform $\widehat u$, defined as
\begin{equation*}
\widehat{u}_{p} = \sum\limits_{m=0}^{N-1}u_{m}e^{-\frac{2\pi i m p}{N}} \Longleftrightarrow u_m = \frac{1}{N}\sum\limits_{p=0}^{N-1}\widehat{u}_{p}e^{\frac{2\pi i m p}{N}},
\end{equation*}
is also $N$-periodic. Then, a very important property of $u\ast v$ is that
\begin{equation}
\label{e:convprop}
(\widehat{u\ast v})_{p} = \widehat{u}_{p}\widehat{v}_{p}, \quad p=0,\ldots,N-1,
\end{equation}
i.e, the discrete Fourier transform of the convolution is just the product of the discrete Fourier transforms. This is a well-known theorem called the discrete convolution theorem, and its proof is straightforward (see for instance \cite{Garcia-Cervera2007}):
\begin{align*}
(\widehat{u\ast v})_{p} & = \sum_{m=0}^{N-1}\left[\sum\limits_{n=0}^{N-1}u_{n}v_{m-n}\right]e^{-\frac{2\pi i m p}{N}}
= \sum_{n=0}^{N-1}u_{n}e^{-\frac{2\pi i n p}{N}}\sum\limits_{m=0}^{N-1}v_{m-n}e^{-\frac{2\pi i (m-n) p}{N}}
	\\
& = \sum_{n=0}^{N-1}u_{n}e^{-\frac{2\pi i n p}{N}}\sum\limits_{q=-n}^{N-n-1}v_{q}e^{-\frac{2\pi i q p}{N}} = \sum_{n=0}^{N-1}u_{n}e^{-\frac{2\pi i n p}{N}}\sum\limits_{q=0}^{N-1}v_{q}e^{-\frac{2\pi i q p}{N}} = \widehat u_p\widehat v_p,
\end{align*}
where we have used that $v_q$ and $e^{-2\pi i qp/N}$ are $N$-periodic.

Observe that we need $\mathcal O(N^2)$ operations to compute directly \eqref{e:defconf}, whereas the computation of a discrete Fourier transform by means of the fast Fourier transform algorithm \cite{FFT} requires $\mathcal O(N\log(N))$ operations. Since we need two fast Fourier transforms (FFT) and one inverse fast Fourier transform (IFFT) to compute \eqref{e:defconf} using \eqref{e:convprop}, the total cost is also $\mathcal O(N\log(N))$ operations, which is much lower than $\mathcal O(N^2)$. Therefore, this technique can be referred to as fast convolution.

On the other hand, given $N_a \in \mathbb N$ and $N_b \in \mathbb N$, if we want to compute a convolution of the form
\begin{equation}
\label{e:nonperiodicconv}
(u\ast v)_{m} \equiv \sum\limits_{n=0}^{N_a-1}u_{n}v_{m-n}, \quad m \in\{0, \ldots, N_b - 1\},
\end{equation}
i.e., only for a finite number of values, and where, e.g., $u$ or $v$ are not periodic, it is not possible to apply directly the fast convolution technique. Instead, after observing that, in order to compute \eqref{e:nonperiodicconv}, we need the values $\{u_0, \ldots, u_{N_a-1}\}$ and $\{v_{-N_a+1}, \ldots, v_{N_b-1}\}$, we take $N\in\mathbb N$, such that $N \ge N_a + N_b - 1$, and extend the sequences $\{u_m\}$ and $\{v_m\}$ to two $N$-periodic sequences as follows:
\begin{equation}
	\label{e:utilden}
	\tilde u_m =
	\begin{cases}
		u_m, & \text{if } m\in\{0, \ldots, N_a-1\},
			\\
		0, & \text{if } N_b \ge 2 \text{ and } m \in \{N_a, \ldots, N-1\},
	\end{cases}
\end{equation}
and
\begin{equation}
	\label{e:vtilden}
	\tilde v_m = \begin{cases}
		v_{m}, & \text{if } m\in\{0,\ldots, N_b-1\}, \\
		0, & \text{if } N \ge N_a + N_b \text{ and } m\in\{N_b,\ldots, N - N_a\}, \\
		v_{m-N}, & \text{if }  N_a \ge  2  \text{ and } m\in\{N-N_a+1,\ldots, N-1\};
	\end{cases}
\end{equation}
for other values of $m$, we simply impose $\tilde u_{m + N} = \tilde u_{m}$ and $\tilde v_{m + N} = \tilde v_{m}$. Then, we have trivially
\begin{equation}
	\label{e:uastv}
	(u\ast v)_{m} \equiv \sum\limits_{n=0}^{N_a-1}u_{n}v_{m-n} = \sum\limits_{n=0}^{N-1}\tilde u_{n}\widetilde v_{m-n} = (\tilde u\ast \tilde v)_{m}, \quad m = 0, \ldots, N_b-1,
\end{equation}
so, in order to obtain $u\ast v$, we extend $u$ and $v$ by means of \eqref{e:utilden} and \eqref{e:vtilden}, then compute $\tilde u\ast\tilde v$ via the fast convolution, and finally keep the first $N_b$ elements of $\tilde u\ast\tilde v$, and simply discard the last $N - N_b$ elements. The global computation cost is again of $\mathcal O(N\log(N))$.

In general, even if it is possible to choose exactly $N = N_a + N_b - 1$, it might be convenient to choose larger values of $N$, because the FFT implementations typically work better when the factorization of $N$ consists of small primes (ideally, powers of $2$). For instance, in \cite{Garcia-Cervera2007}, a similar algorithm was proposed, with $\{u_m\}$ and $\{v_m\}$ consisting of $N$ elements, and their respective extensions $\{\tilde u_m\}$ and $\{\tilde v_m\}$ being $2N$-periodic, instead of $2N-1$-periodic.

\subsection{Numerical approximation of the singular integral \eqref{e:Isj0pi}  by means of the fast convolution}

In order to approximate numerically \eqref{e:Isj0pi}, from which the fractional Laplacian \eqref{e:fraclap0pi2not1} is a particular case, we need to obtain \eqref{e:A1j} and \eqref{e:A2j}, which must be expressed in a suitable way, in order to be able to apply the fast convolution.

In \eqref{e:A1j}, bearing in mind \eqref{e:sgnsnsj}, we replace $s_j$, $\tilde s_n$, $\tilde s_{n+1/2}$ and $\tilde s_{n+1}$ by their definitions, namely $s_j = h_r(2j+1)r$, $\tilde s_n = h_rn$, $\tilde s_{n+1/2} = h_r(n+1/2)$ and $\tilde s_{n+1} = h_r(n+1)$, to obtain
\begin{align*}
A_{1,j} & = \frac{h_r^{\beta + \gamma + 1}}{(\beta + 1)(\gamma + 1)}\sum_{n = 0}^{rN-1}\left(\frac{\sin(h_r(n+1/2))}{h_r(n+1/2)}\right)^\beta\left(\frac{\sin(h_r(n+1/2 - (2j+1)r))}{h_r(n+1/2 - (2j+1)r)}\right)^\gamma
\cr
& \quad\cdot[(n+1)^{\beta+1} - n^{\beta + 1}][\sgn(n+1 - (2j+1)r)|n+1 - (2j+1)r|^{\gamma + 1}
\cr
& \qquad - \sgn(n - (2j+1)r)|n - (2j+1)r|^{\gamma + 1}]f(h_r(n+1/2)).
\end{align*}
In this form, we cannot apply the fast convolution directly, because, whenever $n$ and $j$ appear simultaneously, they do it in the form of $n-2rj$. Therefore, we decompose $n = 2rl + q$, where $l$ and $q$ are nonnegative integers, and $0\le q\le 2r-1$, so $A_{1,j}$ becomes
\begin{align*}
A_{1,j} & = \frac{h_r^{\beta + \gamma + 1}}{(\beta + 1)(\gamma + 1)}\sum_{q = 0}^{2r-1}\sum_{l = 0}^{\lceil N/2\rceil-1}\chi_{[0,rN-1]}(2rl+q)\left(\frac{\sin(h_r(q+1/2 + 2rl))}{h_r(q+1/2 + 2rl)}\right)^\beta
	\cr
	& \quad\cdot\left(\frac{\sin(h_r(q+1/2 - r- 2r(j-l)))}{h_r(q+1/2 - r - 2r(j-l))}\right)^\gamma [(q+1 + 2rl)^{\beta+1} - (q + 2rl)^{\beta + 1}]
	\cr
& \quad \cdot  [\sgn(q+1 - r - 2r(j - l))|q+1 - r - 2r(j-l)|^{\gamma + 1}
	\cr
& \qquad - \sgn(q - r - 2r(j-l))|q - r - 2r(j - l)|^{\gamma + 1}]f(h_r(q+1/2 + 2rl)),
\end{align*}
where
$$
\chi_{[0,rN-1]}(x) =
\begin{cases}
1, & x\in[0,rN-1],
\\
0, & x\not\in[0,rN-1].
\end{cases}
$$
Then, defining
\begin{align*}
K_1(m, q) & = \left(\frac{\sin(h_r(q+1/2 + 2rm))}{h_r(q+1/2 + 2rm)}\right)^\beta
\cr
& \quad\cdot[(q+1 + 2rm)^{\beta+1} - (q + 2rm)^{\beta + 1}]f(h_r(q+1/2+2rm)),
	\cr
L_1(m, q) & = \left(\frac{\sin(h_r(q+1/2 - r- 2rm))}{h_r(q+1/2 - r - 2rm)}\right)^\gamma
 [\sgn(q+1 - r - 2rm)|q+1 - r - 2rm|^{\gamma + 1}
\cr
& \qquad - \sgn(q - r - 2rm)|q - r - 2rm|^{\gamma + 1}],
\end{align*}
we can write
\begin{equation}
\label{e:A1jdoublesum}
A_{1, j} = \frac{h_r^{\beta + \gamma + 1}}{(\beta + 1)(\gamma + 1)}\sum_{q = 0}^{2r-1}\sum_{l = 0}^{\lceil N/2\rceil-1}\chi_{[0,rN-1]}(2rl+q)K_1(l,q)L_1(j-l,q).
\end{equation}
Likewise, we can express \eqref{e:A2j} in an equivalent way to \eqref{e:A1jdoublesum}. In that case, we replace the appearances of $\pi$ by $h_r2rN$, apply \eqref{e:sgnsnsj} again, and make $n$ run between $0$ and $rN-1$, instead of between $rN$ and $2rN-1$, for which we need to replace $n$ by $rN + n$ in the subscripts, i.e., we substitute $\tilde s_n$ by $\tilde s_{rN+n}$, etc. Afterward, we expand $s_j = h_r(2j+1)r$, $\tilde s_{rN+n} = h_r(rN+n)$, $\tilde s_{rN+n+1/2} = h_r(rN+n+1/2)$ and $\tilde s_{rN+n+1} = h_r(rN+n+1)$, obtaining
\begin{align*}
A_{2,j} & = \frac{h_r^{\beta+\gamma+1}}{(\beta+1)(\gamma+1)}\sum_{n = 0}^{rN-1}\left(\frac{\sin(h_r(rN+n+1/2))}{h_r(rN-n-1/2)}\right)^\beta
\cr
& \quad\cdot\left(\frac{\sin(h_r(rN+n+1/2-(2j+1)r)}{h_r(rN+n+1/2- (2j+1)r}\right)^\gamma[(rN-n)^{\beta+1} - (rN-n-1)^{\beta + 1}]
\cr
& \quad
\cdot[\sgn(rN+n+1-(2j+1)r)|rN+n+1-(2j+1)r|^{\gamma + 1} 
\cr
& \qquad - \sgn(rN+n-(2j+1)r)|rN+n-(2j+1)r|^{\gamma + 1}]f(h_r(rN+n+1/2)).
\end{align*}
Again, decomposing $n = 2rl+q$,
\begin{align*}
	A_{2,j} & = \frac{h_r^{\beta+\gamma+1}}{(\beta+1)(\gamma+1)}\sum_{q = 0}^{2r-1}\sum_{l = 0}^{\lceil N/2\rceil-1}\chi_{[0,rN-1]}(2rl+q)
	\cr
	& \quad\cdot\left(\frac{\sin(h_r(rN+q+1/2+2rl))}{h_r(rN-q-1/2-2rl)}\right)^\beta\left(\frac{\sin(h_r(rN+q+1/2-r-2r(j-l))}{h_r(rN+q+1/2-r-2r(j-l)}\right)^\gamma
	\cr
	& \quad \cdot [(rN-q-2rl)^{\beta+1} - (rN-q-1-2rl)^{\beta + 1}]
	\cr
	& \quad
	\cdot[\sgn(rN+q+1-r-2r(j-l))|rN+q+1-r-2r(j-l)|^{\gamma + 1} 
	\cr
	& \qquad - \sgn(rN+q-r-2r(j-l))|rN+q-r-2r(j-l)|^{\gamma + 1}]
	\cr
	& \quad\cdot f(h_r(rN+2rl+q+1/2)).
\end{align*}
Then, defining
\begin{align*}
	K_2(m, q) & = \left(\frac{\sin(h_r(rN+q+1/2+2rm))}{h_r(rN-q-1/2-2rm)}\right)^\beta
	\cr
	& \quad \cdot [(rN-q-2rm)^{\beta+1} - (rN-q-1-2rm)^{\beta + 1}]
f(h_r(rN+2rm+q+1/2)),
	\cr
	L_2(m, q) & = \left(\frac{\sin(h_r(rN+q+1/2-r-2rm)}{h_r(rN+q+1/2-r-2rm}\right)^\gamma
\cr
& \quad
\cdot[\sgn(rN+q+1-r-2rm)|rN+q+1-r-2rm|^{\gamma + 1} 
\cr
& \qquad - \sgn(rN+q-r-2rm)|rN+q-r-2rm|^{\gamma + 1}],
\end{align*}
we can write
\begin{equation}
\label{e:A2jdoublesum}
A_{2, j} = \frac{h_r^{\beta + \gamma + 1}}{(\beta + 1)(\gamma + 1)}\sum_{q = 0}^{2r-1}\sum_{l = 0}^{\lceil N/2\rceil-1}\chi_{[0,rN-1]}(2rl+q)K_2(l,q)L_2(j-l,q).
\end{equation}
Note that, obviously, $K_1$, $L_1$, $K_2$ and $L_2$ depend on more parameters than $m$ and $q$, but only the variables $m$ and $q$ are relevant for the implementation of the algorithm. In order to proceed further, we bear in mind that, when $N$ is even, $\lceil N/2\rceil = N/2$, whereas, when $N$ is odd, $\lceil N/2\rceil = (N+1)/2$. Hence, when $N$ is even, for the range of values of $l$ and $q$ considered,  $\chi_{[0,rN-1]}(2rl+q) \equiv 1$ and can be omitted, whereas, when $N$ is odd, $\chi_{[0,rN-1]}(2rl+q) \equiv 0$ only when $l = \lceil N/2\rceil-1 = (N-1)/2$ and $q \ge r$. Taking this into account, we define
\begin{align*}
A_{1, j}^q & =
\left\{
\begin{aligned}
	& \sum_{l = 0}^{\lceil N/2\rceil-1}K_1(l,q)L_1(j-l,q), & & \text{if } q\in\{0, \ldots, r-1\},
	\cr
	& \sum_{l = 0}^{\lfloor N/2\rfloor-1}K_1(l,q)L_1(j-l,q), & & \text{if } q\in\{r, \ldots, 2r-1\},
\end{aligned}
\right.
\cr
A_{2, j}^q & =
\left\{
\begin{aligned}
	& \sum_{l = 0}^{\lceil N/2\rceil-1}K_2(l,q)L_2(j-l,q), & & \text{if } q\in\{0, \ldots, r-1\},
	\cr
	& \sum_{l = 0}^{\lfloor N/2\rfloor-1}K_2(l,q)L_2(j-l,q), & & \text{if } q\in\{r, \ldots, 2r-1\},
\end{aligned}
\right.
\end{align*}
which match exactly the structure of \eqref{e:uastv}, where $N_a$ in \eqref{e:uastv} is now $\lceil N/2\rceil$, when $0\le q\le r - 1$, and $\lfloor N/2\rfloor$, when $r \le q\le2r-1$; and $N_b$ in \eqref{e:uastv} is now $N$, and therefore, can be computed efficiently by means of the fast convolution technique. Then, the approximation of \eqref{e:Isj0pi} is expressed as
\begin{equation}
\label{e:Isjapprox}
I(s_j) \approx A_{1, j} + A_{2, j} = \frac{h_r^{\beta + \gamma + 1}}{(\beta + 1)(\gamma + 1)}\sum_{q = 0}^{2r-1}(A_{1, j}^q + A_{2, j}^q), \quad j\in\{0, \ldots, N-1\}.
\end{equation}
From an implementation point of view, it is convenient to work with matrices. More precisely, we create four matrices $\tilde{\mathbf K}_1$, $\tilde{\mathbf K}_2$, $\tilde{\mathbf L}_1$, and $\tilde{\mathbf L}_2$, consisting of $2r$ columns and at least $\lceil N/2 \rceil + N - 1$ rows, which is enough for all the values of $q$. Then, taking $q = 0$, we generate the first column of $\tilde{\mathbf K}_1$ and the first column of $\tilde{\mathbf K}_2$ using \eqref{e:utilden}, and the first column of $\tilde{\mathbf L}_1$ and the first column of $\tilde{\mathbf L}_2$ using \eqref{e:utilden}. Likewise, $q = 1$ is chosen to generate the second column of each matrix, $q = 2$, to generate the third one, and so on, until $q = 2r-1$, to generate the last one. At this point, the FFT is applied columnwise to the four matrices, to obtain respectively $\hat{\mathbf K}_1$, $\hat{\mathbf K}_2$, $\hat{\mathbf L}_1$, $\hat{\mathbf L}_2$, and, afterward, the IFFT can be applied columnwise to $\hat{\mathbf K}_1 \circ \hat{\mathbf L}_1 + \hat{\mathbf K}_2 \circ \hat{\mathbf L}_2$ (where $\circ$ denotes the pointwise or Hadamard product), to obtain a matrix $\mathbf A$;  keeping only the first $N$ rows, and adding the elements of each row, we obtain a vector of $N$ elements. However, since the IFFT is a linear operator, we have followed a slightly faster, but equivalent approach: we have added the elements of each row of $\hat{\mathbf K}_1 \circ \hat{\mathbf L}_1 + \hat{\mathbf K}_2 \circ \hat{\mathbf L}_2$, applied the IFFT to the resulting vector, to obtain a new vector $\mathbf A$, and kept the first $N$ rows of $\mathbf A$, to get a vector of $N$ elements. Multiplying that last vector by $h_r^{\beta + \gamma + 1} / ((\beta + 1)(\gamma + 1))$, we get simultaneously the approximation of \eqref{e:Isjapprox}, for $j\in\{0, \ldots, N-1\}$.

With respect to the computational cost, the most expensive part is the computation of the FFT and IFFT. Since each column has a length of $\mathcal O(N)$, the cost of applying the FFT or IFFT to any single column is of $\mathcal O(N\log(N))$ operations, but we are applying the FFT to the columns of four matrices of $2r$ columns, and the IFFT to one column vector, so the global computational cost is of $\mathcal O(rN\log(N))$ operations.

\section{Implementation and numerical experiments}

\label{s:implementation}

In this section we present an implementation of the algorithm in \textsc{Matlab}  \cite{matlab}, which matches exactly all the steps that we have described (although, our implementation does not exclude of course further optimizations).

\subsection{Implementation in \textsc{Matlab}}

We have defined a variable \verb"nrows", to store the number of rows of $\tilde{\mathbf K}_1$, $\tilde{\mathbf K}_2$, $\tilde{\mathbf L}_1$, and $\tilde{\mathbf L}_2$. Even if this value must be at least $\lceil N/2 \rceil + N - 1$, we have chosen it to be the lowest natural number that is a power of $2$ and is larger than or equal to both $\lceil N/2 \rceil + N - 1$ and $2$ (because, otherwise, when $N=1$, $\texttt{nrows}=1$ returns wrong results), i.e., instead of \verb"nrows = ceil(N/2) + N - 1;", we have typed \verb"nrows=max(2^ceil(log2(ceil(N/2)+N-1)),2);". Indeed, when $N$ is very large, this improves noticeably the global execution speed, and this is particularly true when $\lceil N/2\rceil+N-1$ is a large prime number, because the implementations of the FFT are usually less efficient in that case. For instance, if $N = 11184840$, then $\lceil N/2\rceil + N - 1 = 16777259$ is a prime number, but the smallest power of $2$ larger than or equal to $\lceil N/2\rceil + N - 1$, i.e., $\max\{2^{\lceil\log_2(\lceil N/2\rceil+N-1)\rceil},2\} = 33554432$, is approximately twice as large as $16777259$. However, the commands \verb"tic, A = fft(1:16777259); toc" and \verb"tic, A = fft(1:33554432); toc" take respectively about 3.6 seconds and 0.7 seconds to execute, i.e., the latter takes less than a fifth of the elapsed time of the former.

On the other hand, we pass the values $f(\tilde s_{n+1/2}) $ in a column vector $\mathbf F$, rather than the function $f(s)$ itself, which has no remarkable effect in the execution time, and is very useful if for instance the analytical form of $f(s)$ is not known.

\begin{lstlisting}[style=Matlab-editor, language=Matlab, basicstyle={\footnotesize\ttfamily}, caption = {Numerical approximation of the singular integral \eqref{e:Isj0pi}}]
% Approximate numerically
% I(s_j)=\int_0^\pi\sin^\beta(\eta)|\sin(\eta-s_j)|^\gamma f(\eta)d\eta
% The parameters N, r, b, and g denote respectively $N$, $r$, $\beta$
% and $\gamma$, and the parameter F is a vector containing the values
% $f(\tilde s_{n+1/2}$
function I=singularintegral(N,r,b,g,F)
% nrows is the number of rows of tildeK1, tildeL1 tildeK2, tildeL2
nrows=max(2^ceil(log2(ceil(N/2)+N-1)),2);
% Compute $h_r=\pi/(2rN)$
hr=pi/(2*r*N);
% Define the functions $K_1(m,q)$, $K_2(m,q)$, $L_1(m,q)$ and $L_2(m,q)$ 
K1=@(m,q)(sin(hr*(q+1/2+2*r*m))./(hr*(q+1/2+2*r*m))).^b...
    .*((q+1+2*r*m).^(b+1)-(q+2*r*m).^(b+1)).*F(q+2*r*m+1);
L1=@(m,q)(sin(hr*(q+1/2-r-2*r*m))./(hr*(q+1/2-r-2*r*m))).^g...
    .*(sign(q+1-r-2*r*m).*abs(q+1-r-2*r*m).^(g+1)...
        -sign(q-r-2*r*m).*abs(q-r-2*r*m).^(g+1));
K2=@(m,q)(sin(hr*(r*N+q+1/2+2*r*m))./(hr*(r*N-q-1/2-2*r*m))).^b...
    .*((r*N-q-2*r*m).^(b+1)-(r*N-q-1-2*r*m).^(b+1)).*F(r*N+q+2*r*m+1);
L2=@(m,q)(sin(hr*(r*N+q+1/2-r-2*r*m))./(hr*(r*N+q+1/2-r-2*r*m))).^g...
    .*(sign(r*N+q+1-r-2*r*m).*abs(r*N+q+1-r-2*r*m).^(g+1)...
        -sign(r*N+q-r-2*r*m).*abs(r*N+q-r-2*r*m).^(g+1));
% Define the matrices $\tilde{\mathbf K}_1$, $\tilde{\mathbf K}_2$,
% $\tilde{\mathbf L}_1$, and $\tilde{\mathbf L}_2$
tildeK1=zeros(nrows,2*r);
tildeL1=zeros(nrows,2*r);
tildeK2=zeros(nrows,2*r);
tildeL2=zeros(nrows,2*r);
for q=0:r-1
    tildeK1(1:ceil(N/2),q+1)=K1((0:ceil(N/2)-1)',q);
    tildeK2(1:ceil(N/2),q+1)=K2((0:ceil(N/2)-1)',q);
    tildeL1([1:N end-ceil(N/2)+2:end],q+1)...
        =L1([0:N-1 -ceil(N/2)+1:-1]',q);
    tildeL2([1:N end-ceil(N/2)+2:end],q+1)...
        =L2([0:N-1 -ceil(N/2)+1:-1]',q);
end
for q=r:2*r-1
    tildeK1(1:floor(N/2),q+1)=K1((0:floor(N/2)-1)',q);
    tildeK2(1:floor(N/2),q+1)=K2((0:floor(N/2)-1)',q);
    tildeL1([1:N end-floor(N/2)+2:end],q+1)...
        =L1([0:N-1 -floor(N/2)+1:-1]',q);
    tildeL2([1:N end-floor(N/2)+2:end],q+1)...
        =L2([0:N-1 -floor(N/2)+1:-1]',q);
end
% Apply the fast convolution, to get the vector $\mathbf A$
A=ifft(sum(fft(tildeK1).*fft(tildeL1)+fft(tildeK2).*fft(tildeL2),2));
% Return the vector $\mathbf I$ containing the numerical approximation
% of the values $I(s_j)$
I=(hr^(b+g+1)/((b+1)*(g+1)))*A(1:N);
\end{lstlisting}

Note that it is straightforward to use this code to approximate numerically the fractional Laplacian \eqref{e:fraclap0pi2not1}. Indeed, we only need to define $f(s) = \sin(s)u_{ss}(s) + 2\cos(s)u_{s}(s)$, take $\beta = \alpha$ and $\gamma = 1 - \alpha$, invoke the function \verb"singularintegral", and multiply each component of the result \verb"I" by its corresponding factor, which we have simplified slightly, bearing in mind the well-known identities $\Gamma(z)\Gamma(1-z) = \pi / \sin(\pi z)$, for $z\not\in\mathbb Z$, and $\Gamma(z)\Gamma(z+1/2) = 2^{1-2z}\sqrt\pi\Gamma(2z)$, and the fact that we evaluate $(-\Delta)^{\alpha/2}u(s)$ at points $s_j\in(0,\pi)$:
$$
\frac{c_{\alpha}|\sin(s_j)|^{\alpha-1}}{L^\alpha\alpha(1-\alpha)} = \alpha\frac{2^{\alpha-1}\Gamma(1/2+\alpha/2)}{\sqrt{\pi}\Gamma(1-\alpha/2)}\frac{|\sin(s_j)|^{\alpha-1}}{L^\alpha\alpha(1-\alpha)} = \frac{\sin^{\alpha-1}(s_j)}{L^\alpha2\Gamma(2 - \alpha)\cos(\pi\alpha/2)}.
$$
Although the next section will be specifically devoted to the numerical experiments, the following code approximates numerically \eqref{e:fraclap0pi2not1} for $u(s) = e^{2is}$ and $L = 1$ (i.e., for $f(s) = \sin(s)u_{ss}(s) + 2\cos(s)u_{s}(s) = (-4\sin(s) + 2i\cos(s))e^{2is}$), comparing the numerical result \verb"fraclapnum" with the exact result \verb"fraclap" given in Theorem 2.1 of \cite{cayamacuestadelahoz2020},
for $n = 1$:
\begin{equation}
\label{e:deltaa2e2i}
(-\Delta)^{\alpha/2}e^{2is} \equiv (-\Delta)^{\alpha/2}\left(\frac{ix - 1}{ix + 1}\right) = -\frac{2\Gamma(1+\alpha)}{(ix+1)^{1+\alpha}} =  -\frac{2\Gamma(1+\alpha)}{(i\cot(s)+1)^{1+\alpha}}.
\end{equation}
In this specific example, we have chosen $\alpha = 1.3$, $N = 10000019$ and $r = 1$.  Note that the choice of $N$ that we have made (i.e, the smallest prime number larger than ten millions) serves to illustrate that the code works equally well, even if $N$ does not have a particular factorization.

The code takes about 14 seconds to execute, and the discrete $L^2$ and $L^\infty$ norms of the error \verb"fraclapnum - fraclap" are respectively  $5.2215\times10^{-11}$ and $6.9554\times10^{-14}$. In our opinion, these results are very remarkable, especially taking into account the very large (prime) number $N$ of points considered.

\begin{lstlisting}[style=Matlab-editor, language=Matlab, basicstyle={\footnotesize\ttfamily}, caption = {Numerical approximation of $(-\Delta)^{\alpha/2}e^{2is}\equiv(-\Delta)^{\alpha/2}((ix-1)/(ix+1))$}]
% Approximate numerically the fractional Laplacian of 
% $u(x)=(ix-1)/(ix+1)$. We work with $u(s)\equiv u(\cot(s))=e^{2is}$
% The variables N, r and a denote respectively $N$, $r$ and $\alpha$
clear
tic
N=10000019;
r=1;
a=1.3;
% Define the function
% $f(s)=\sin(s)u_{ss}(s)+2\cos(s)u_{s}(s)=(-4\sin(s)+4i\cos(s))e^{2is}$
f=@(s)(-4*sin(s)+4i*cos(s)).*exp(2i*s);
% F is the vector containing the values $f(\tilde s_{n+1/2}$
F=f(pi*((0:2*r*N-1)'+1/2)/(2*r*N));
% Invoke the Matlab function singularintegral, to obtain the vector I
% containing the numerical approximation of the values $I(s_j)$
I=singularintegral(N,r,a,1-a,F);
% sj is the vector containing the nodes $s_j$
sj=pi*((0:N-1)'+.5)/N;
% fraclapnum and fraclap store respectively the numerical approximation
% and the exact value of the fractional Laplacian of $u(x)$
fraclapnum=(1/(2*gamma(2-a)*cos(pi*a/2)))*sin(sj).^(a-1).*I;
fraclap=-2*gamma(1+a)./(1i*cot(sj)+1).^(1+a);
norm(fraclapnum-fraclap) % Error in discrete $L^2$ norm
norm(fraclapnum-fraclap,"inf") % Error in discrete $L^\infty$ norm
toc % Elapsed time
\end{lstlisting}

\subsubsection{Implementation with $u_s(s)$ and $u_{ss}(s)$ not explicitly given, and $L\not=1$}

\label{s:usussnotexplicit}

Obviously, if possible, it is better to work with the exact analytical form of $f(s)$. However, if for instance we only know the values of $u(x)$ at $x = x_j = L\cot(s_j)$, we are still able to approximate all the values of $f(\tilde s_n)$ and construct the column vector $\mathbf F$, by means of, e.g., a pseudospectral method, following the ideas in \cite{cayamacuestadelahoz2021}. More precisely, after denoting $u(s) \equiv u(x(s))$, we extend $u(s)$ from $s\in[0,\pi]$ to $s\in[0,2\pi]$, which can be done by, e.g., an even extension at $s = \pi$, i.e., by imposing that $f(\pi+s)\equiv f(\pi - s)$, which guarantees that $u(s)$ is continuous and periodic over $s\in[0,2\pi]$. Thus, we represent
\begin{equation}
\label{e:fourierus}
u(s) \approx \sum_{k=-N}^{N-1}\hat{u}(k)e^{iks}, \quad s\in[0,\ 2\pi].
\end{equation}
Evaluating this equation at $s = s_j$ (where the nodes $s_j$ are still defined by \eqref{e:sj}, but now for $j\in\{0, \ldots, 2N-1\}$), and imposing the equality at those nodes (which is at the basis of a pseudospectral method), we get
\begin{align}
\label{e:usjfourier}
	u(s_j) & \equiv \sum\limits_{k=-N}^{N-1}\hat{u}(k)e^{iks_j} = \sum\limits_{k=-N}^{N-1}\hat{u}(k)e^{ik\pi(2j+1)/(2N)}
	\cr
	& = \sum\limits_{k=0}^{N-1}\left[\hat{u}(k)e^{ik\pi/(2N)}\right]e^{2ijk\pi/(2N)}
	\cr
	& \quad + \sum\limits_{k=N}^{2N-1}\left[\hat{u}(k-2N)e^{i(k-2N)\pi/(2N)}\right]e^{2ijk\pi/(2N)},
\end{align}
\noindent so the values $\{u(s_j)\}$ are precisely the inverse discrete Fourier transform of
\begin{align*}
\big\{\hat u(0), & \hat{u}(1)e^{i\pi/(2N)}, \ldots, \hat{u}(N-1)e^{i(N-1)\pi/(2N)}, 
\hat{u}(-N)e^{-iN\pi/(2N)}, \ldots, \hat{u}(-1)e^{-i\pi/(2N)}\big\};
\end{align*}
and, conversely, the Fourier coefficients $\hat u(k)$ are given by the discrete Fourier transform of $\{u(s_0), \ldots, u(s_{N-1})\}$, multiplied by $e^{-ik\pi/(2N)}$, for $k \in\{ 0, \ldots, N-1\}\cup\{ -N, \ldots, -1\}$:
\begin{equation*}
\hat{u}(k) \equiv \frac{e^{-ik\pi/(2N)}}{2N}\sum_{j=0}^{2N-1}u(s_{j})e^{-2ijk\pi/(2N)}.
\end{equation*}
At this point, we apply the so-called Krasny filter \cite{krasny}, i.e., we set to zero all the Fourier coefficients $\hat u(k)$ whose modulus is smaller than a fixed threshold, which in this paper is the epsilon of the machine. Then, from \eqref{e:fourierus}:
\begin{equation}
\label{e:ususs}
u_s(s) \approx \sum_{k=-N}^{N-1}ik\hat{u}(k)e^{iks}, \quad u_{ss}(s) \approx \sum_{k=-N}^{N-1}(ik)^2\hat{u}(k)e^{iks}, \quad s\in[0,\ 2\pi].
\end{equation}
Recall that in $[0, \pi]$, there are $2rN$ points $\tilde s_{n+1/2} = h_r(n+1/2) = \pi(2n+1)/(4rN)$, with $n \in\{0, \ldots, 2rN\}$, so in $[0,2\pi]$, we take $n \in\{0, \ldots, 4rN\}$. Therefore, in order to approximate $u_s(s)$ and $u_{ss}(s)$ at the $4rN$ values of $\tilde s_{n+1/2} = h_r(n+1/2)$, we extend the definition of $\hat{u}(k)$ as follows:
$$
\hat{\tilde u}(k) \equiv
\begin{cases}
\hat{u}(k), & \text{if } k\in\{-N, \ldots, N-1\},
	\cr
0, & \text{if } k\in\{-2rN, \ldots, -N-1\} \cup \{N, \ldots, 2rN-1\}.
\end{cases}
$$
Then, reasoning as in \eqref{e:usjfourier}, we impose the equalities in \eqref{e:ususs} at $\tilde s_{n+1/2}$, to get
\begin{align*}
u_s(\tilde s_{n+1/2}) & \equiv \sum_{k=-N}^{N-1}ik\hat{u}(k)e^{ik\tilde s_{n+1/2}} = \sum_{k=-2rN}^{2rN-1}ik\hat{\tilde u}(k)e^{ik\pi(2n+1)/(4rN)}
	\cr
& = \sum\limits_{k=0}^{2rN-1}\left[ik\hat{\tilde u}(k)e^{ik\pi/(4rN)}\right]e^{2ink\pi/(4rN)}
\cr
& \quad + \sum\limits_{k=2rN}^{4rN-1}\left[i(k-4rN)\hat{\tilde u}(k-4rN)e^{i(k-4rN)\pi/(4rN)}\right]e^{2ink\pi/(4rN)},
\end{align*}
and
\begin{align*}
u_{ss}(\tilde s_{n+1/2}) & \equiv \sum_{k=-N}^{N-1}(ik)^2\hat{u}(k)e^{ik\tilde s_{n+1/2}} = \sum_{k=-2rN}^{2rN-1}(ik)^2\hat{\tilde u}(k)e^{ik\pi(2n+1)/(4rN)}
\cr
& = \sum\limits_{k=0}^{2rN-1}\left[(ik)^2\hat{\tilde u}(k)e^{ik\pi/(4rN)}\right]e^{2ink\pi/(4rN)}
\cr
& \quad + \sum\limits_{k=2rN}^{4rN-1}\left[(i(k-4rN))^2\hat{\tilde u}(k-4rN)e^{i(k-4rN)\pi/(4rN)}\right]e^{2ink\pi/(4rN)},
\end{align*}
\noindent so the values $\{u_s(\tilde s_{n+1/2})\}$ and $\{u_{ss}(\tilde s_{n+1/2})\}$ are respectively the $4Nr$-element inverse discrete Fourier transforms of
\begin{align*}
	\big\{0, & i\hat{\tilde u}(1)e^{i\pi/(4rN)}, \ldots, i(2Nr-1)\hat{\tilde u}(2Nr-1)e^{i(2Nr-1)\pi/(4rN)}, 
	\cr
	& {-i2rN}\hat{\tilde u}(-2rN)e^{-i2rN\pi/(4rN)}, \ldots, -i\hat{\tilde u}(-1)e^{-i\pi/(4rN)}\big\}
\end{align*}
and
\begin{align*}
	\big\{0, & {-\hat{\tilde u}(1)}e^{i\pi/(4rN)}, \ldots, -(2Nr-1)^2\hat{\tilde u}(2Nr-1)e^{i(2Nr-1)\pi/(4rN)}, 
	\cr
	& {-(2rN)^2}\hat{\tilde u}(-2rN)e^{-i2rN\pi/(4rN)}, \ldots, -\hat{\tilde u}(-1)e^{-i\pi/(4rN)}\big\},
\end{align*}
which require $\mathcal O(Nr\log(Nr))$ operations. Finally, we store in the column vector $\mathbf F$ the first $2rN$ values of $f(\tilde s_{n+1/2}) = \sin(\tilde s_{n+1/2})u_{ss}(\tilde s_{n+1/2}) + 2\cos(\tilde s_{n+1/2})u_{s}(\tilde s_{n+1/2})$, i.e., for $n\in\{0, \ldots, 2rN-1\}$.

The implementation in \textsc{Matlab} is again straightforward. In the following example, we approximate numerically the fractional Laplacian of the Gauss error function $\erf(x)$:
\begin{align}
\label{e:fraclaperf}
\erf(x) & = \frac{2}{\sqrt\pi}\int_{0}^{x}e^{-y^2}dy
 \Longrightarrow (-\Delta)^{\alpha/2}\erf(x) = \frac{2^{1+\alpha}}{\pi}\Gamma\left(\frac{1+\alpha}{2}\right)x\,{}_1F_1\left(\frac{1+\alpha}{2}; \frac32; -x^2\right),
\end{align}
where ${}_1F_1$ is Kummer's hypergeometric confluent function. Note that this formula is immediately obtained after introducing \eqref{e:fraclapl1} in \textsc{Mathematica} \cite{mathematica}, where now $u(x) = \erf(x)$ and $u_x(x) = (2/\sqrt\pi)e^{-x^2}$. More precisely, we just type
\begin{verbatim}
Integrate[(Exp[-(x-y)^2]-Exp[-(x+y)^2])/y^a,{y,0,Infinity}]
\end{verbatim}
multiply the result by $(c_\alpha/\alpha)(2/\sqrt\pi)$, and simplify the constants.

We have chosen now $\alpha = 0.9$, $L = 2.1$, $N = 2^{20} = 1048576$, $r = 8$. The numerical method proposed in this paper takes slightly less than 10 seconds to execute, whereas the exact solution (which requires the computation of ${}_1F_1$) needs 2442 seconds. On the other hand, the discrete $L^2$ and $L^\infty$ norms of the error \verb"fraclapnum - fraclap" are respectively $8.1118\times10^{-12}$ and $2.7311\times10^{-14}$. 

With respect to $L$,  the value $L = 2.1$ in this example serves mainly to illustrate the implementation of a case with $L\not=1$, and there are other possible choices of $L$ that give equally good results. Indeed, as said in \cite{cayamacuestadelahoz2021}, although there are some theoretical results \cite{Boyd1982}, the optimal choice of $L$ is a complicate one, because it depends on many factors: number of points, class of functions, type of problem (in the case of the fractional Laplacian, it may depend on  $\alpha$), etc. However, a good working rule of thumb seems to be that the absolute value of a given function at the extreme grid points is smaller than a given threshold.

\begin{lstlisting}[style=Matlab-editor, language=Matlab, basicstyle={\footnotesize\ttfamily}, caption={Numerical approximation of $(-\Delta)^{\alpha/2}\erf(x)$, with $L\not=1$}]
% Approximate numerically the fractional Laplacian of $u(x)=\erf(x)$
% ($\erf$ is defined in LaTeX by $\DeclareMathOperator{\erf}{erf}$)
% The variables N, r, a and L denote respectively $N$, $r$, $\alpha$
% and $L$
clear
N=1048576;
r=8;
a=0.9;
L=2.1;
% sj is the vector containing the nodes $s_j$
sj=pi*((0:N-1)'+.5)/N;
% xj is the vector containing the nodes $x_j=\cot(s_j)$
xj=L*cot(sj);
tic
% uxj is the vector containing the values $u(x_j)=\erf(x_j)
uxj=erf(xj);
% Compute the FFT of an even extension of $\erf$
% u_ is the vector containing the values $\hat{u}(k)$
u_=fft([uxj;uxj(end:-1:1)]).*exp(-1i*pi*[0:N-1 -N:-1]'/(2*N));
u_(abs(u_)/(2*N)<eps)=0; % Apply the Krasny filter
% utilde_ is the vector containing the values $\hat{\tilde u}(k)$
utilde_=2*r*[u_(1:N);zeros((4*r-2)*N,1);u_(N+1:2*N)]...
    .*exp(1i*pi*[0:2*r*N-1 -2*r*N:-1]'/(4*r*N));
% ik is the vector containing $ik$, with
% $k\in\{0,\ldots,2rN-1\}\cup\{-2rN,\ldots,-1\}$
ik=1i*[0:2*r*N-1 -2*r*N:-1]';
% tildesn12 is the vector containing the values $\tilde s_{n+1/2}$
tildesn12=pi*((0:2*r*N-1)'+.5)/(2*r*N);
% ustildesn12 is the vector containing the numerical approximation of
% the values $u_s(\tilde s_{n+1/2})$
ustildesn12=ifft(ik.*utilde_);
% usstildesn12 is the vector containing the numerical approximation of
% the values $u_{ss}(\tilde s_{n+1/2})$
usstildesn12=ifft(ik.^2.*utilde_);
% F is the vector containing the numerical approximation of the values
% $f(\tilde s_{n+1/2}$
F=sin(tildesn12).*usstildesn12(1:2*r*N)...
    +2*cos(tildesn12).*ustildesn12(1:2*r*N);
% Invoke the Matlab function singularintegral, to obtain the vector I
% containing the numerical approximation of the values $I(s_j)$
I=singularintegral(N,r,a,1-a,F);
% fraclapnum and fraclap store respectively the numerical approximation
% and the exact value of the fractional Laplacian of $u(x)$
fraclapnum=(1/(L^a*2*gamma(2-a)*cos(pi*a/2)))*sin(sj).^(a-1).*I;
toc,tic % Elapsed time
fraclap=(2^(1+a)*gamma((1+a)/2)/pi)*xj.*hypergeom((1+a)/2,3/2,-xj.^2);
toc % Elapsed time
norm(fraclapnum-fraclap) % Error in discrete $L^2$ norm
norm(fraclapnum-fraclap,"inf") % Error in discrete $L^\infty$ norm
\end{lstlisting}

\section{Numerical experiments}

\label{s:numerical}

We have applied the numerical method to approximate $(-\Delta)^{\alpha/2}(e^{i2s}) = (-\Delta)^{\alpha/2}((ix-1)/(ix+1)) $ and $(-\Delta)^{\alpha/2}\erf(x)$ numerically for different values of $\alpha$, $N$ and $r$, comparing the results with \eqref{e:deltaa2e2i} and \eqref{e:fraclaperf}, respectively. For that purpose, we have run batch versions of the codes offered in Section \ref{s:implementation}; therefore, we have given explicitly $f(s) = \sin(s)u_{ss}(s) + 2\cos(s)u_{s}(s)$ in the former case, and approximated it (choosing $L\not=1$) in the latter case. In order to measure the accuracy of the results, we use the discrete $L^2$ norm of the error (which shows the order of convergence in a slightly clearer way than the discrete $L^\infty$ norm) divided by the square root of the number of points; this normalization of the discrete $L^2$ norm is customary, because it enables to compare vectors of different lengths:
\begin{align}
\label{e:ENra}
E_N^r(\alpha) & = \frac{1}{\sqrt N}\left\|(-\Delta)^{\alpha/2}u - (-\Delta)_{num}^{\alpha/2}u\right\|_2 
	\cr
& = \left[\frac{1}{N}\sum_{j=0}^{N-1}\left|(-\Delta)^{\alpha/2}u(x(s_j)) - (-\Delta)_{\rm num}^{\alpha/2}u(x(s_j)) \right|^2\right]^{\frac12},
\end{align}
where $(-\Delta)_{\rm num}^{\alpha/2}$ denotes the numerical approximation of the fractional Laplacian.

According to Theorem \ref{theo:int0pibg} and Corollary \ref{corol:fraclap}, we have $E_N^r(\alpha) = \mathcal O(1/r^2)$, for all $N$ and for all $\alpha$; in particular, if we multiply $r$ by $2$, the error norm \eqref{e:ENra} will be reduced approximately by $4$. Therefore, in order to confirm this numerically, we have plotted in Figure~\ref{f:orderofconvergence} $\log_2(E_N^{2r}(\alpha) / E_N^r(\alpha))$, for $N = 128$, $r \in\{2^0, 2^1, \ldots, 2^9\}$ and $\alpha \in\{0.01, \ldots 0.99\}\cup\{1.01, \ldots, 1.99\}$. In this and all the other figures in this section, the experiments corresponding to $u(x) = (ix-1)/(ix+1)$ appear on the left-hand side, and those corresponding to $u(x) = \erf(x)$, where we have taking systematically $L = 2.1$, on the right hand side. As we can see, the curves tend to $2$ as $r$ increases, and indeed, in both cases, the curve corresponding to the largest value of $r$, namely, $r = 2^9$, is the closest one to $2$, which corroborates the second order of convergence. Other choices of $N$ and $L$ yield similar results.
\begin{figure}[!htbp]
	\centering
	\includegraphics[width=0.5\textwidth, clip=true]{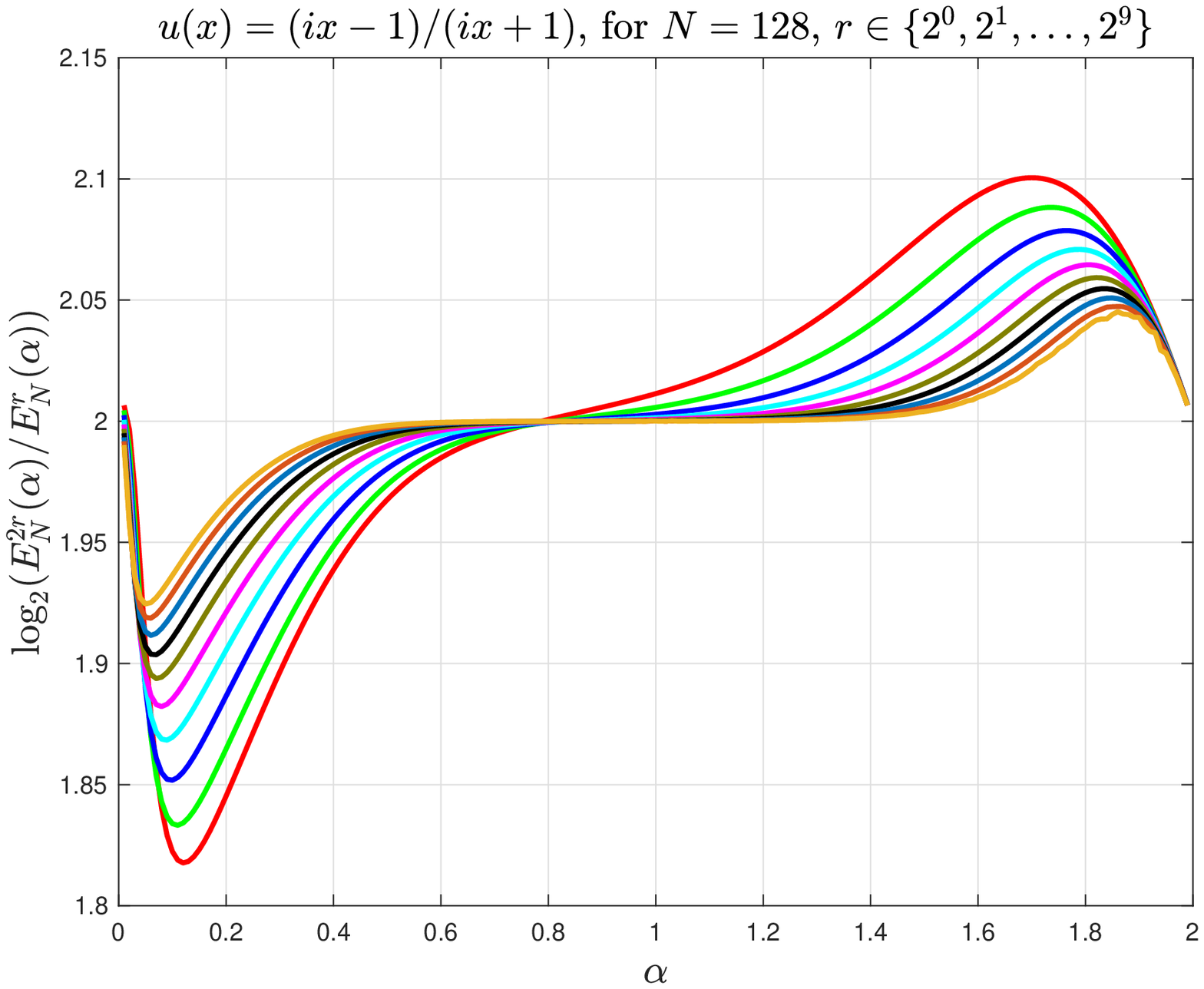}\includegraphics[width=0.5\textwidth, clip=true]{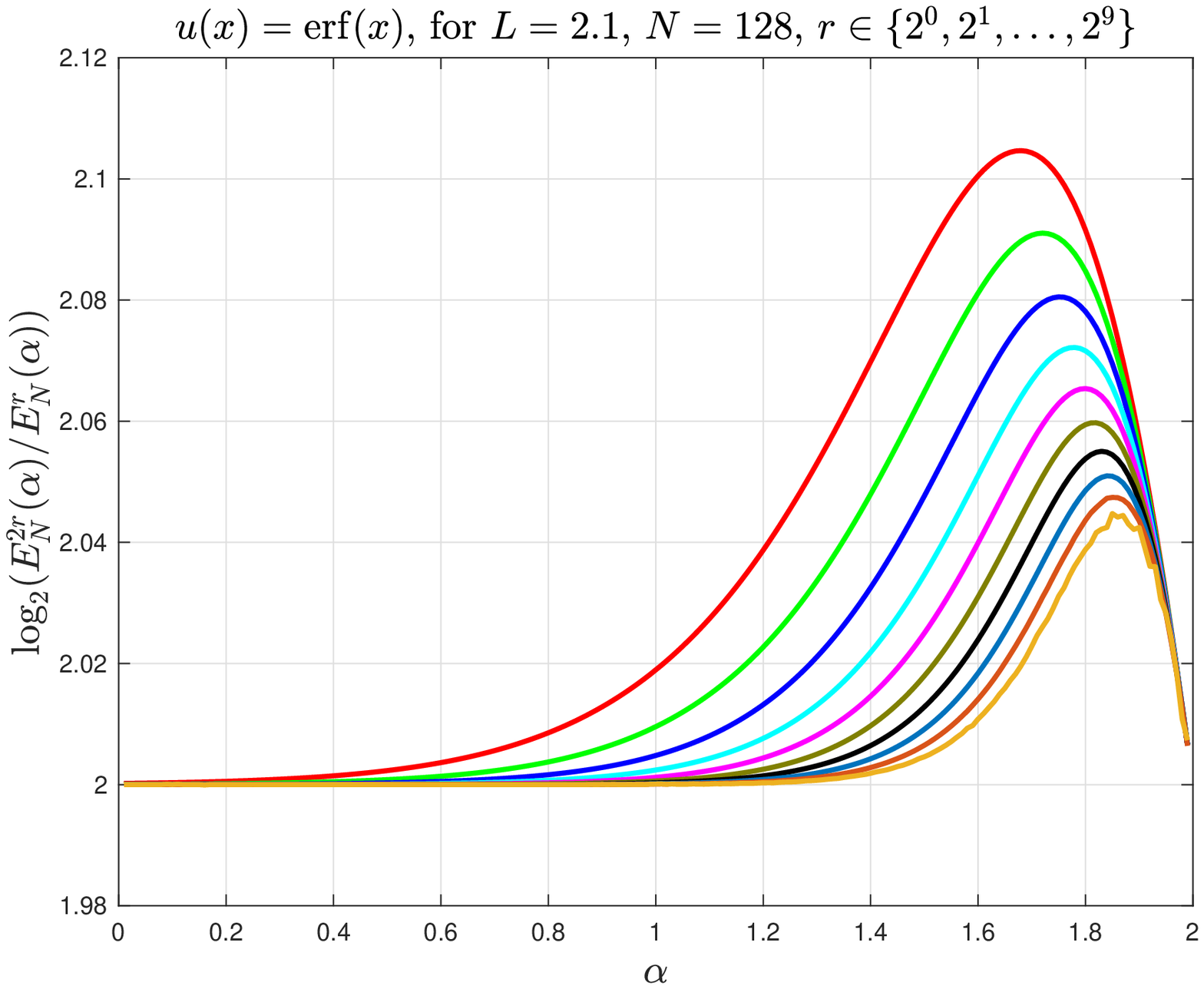}
	\includegraphics[width=\textwidth, clip=true]{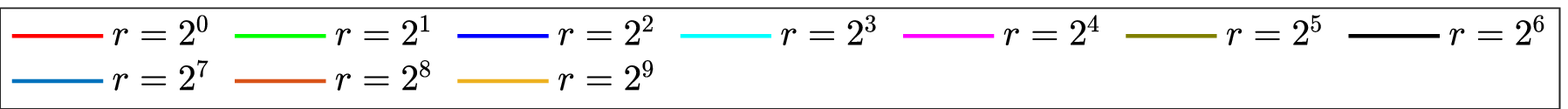}	
	\caption{Numerical measure of the order of convergence of the method. $E_N^r(\alpha)$ is given by \eqref{e:ENra}, using \eqref{e:deltaa2e2i} (left) and \eqref{e:fraclaperf} (right).}
	\label{f:orderofconvergence}
\end{figure}

In Figure \ref{f:errors}, we have plotted in semilogarithmic scale the norms of the errors $E_N^r(\alpha)$ corresponding to the experiments in Figure \ref{f:orderofconvergence}, for $ r\in\{2^0, 2^1, \ldots, 2^{10}\}$. The fact that $E_N^{2r}(\alpha) \approx E_N^{r}(\alpha)/4$ implies that the curves in the same subfigure are roughly parallel and equidistant to each other (recall that we are taking logarithms in the ordinate axis). Note also that the errors corresponding to $u(x) = \erf(x)$ are slightly better; in fact, the errors do depend in general on the characteristics of the function considered, and on the choice of $L$ made, but we will get convergence of order two in all cases. 

\begin{figure}[!htbp]
	\centering
	\includegraphics[width=0.5\textwidth, clip=true]{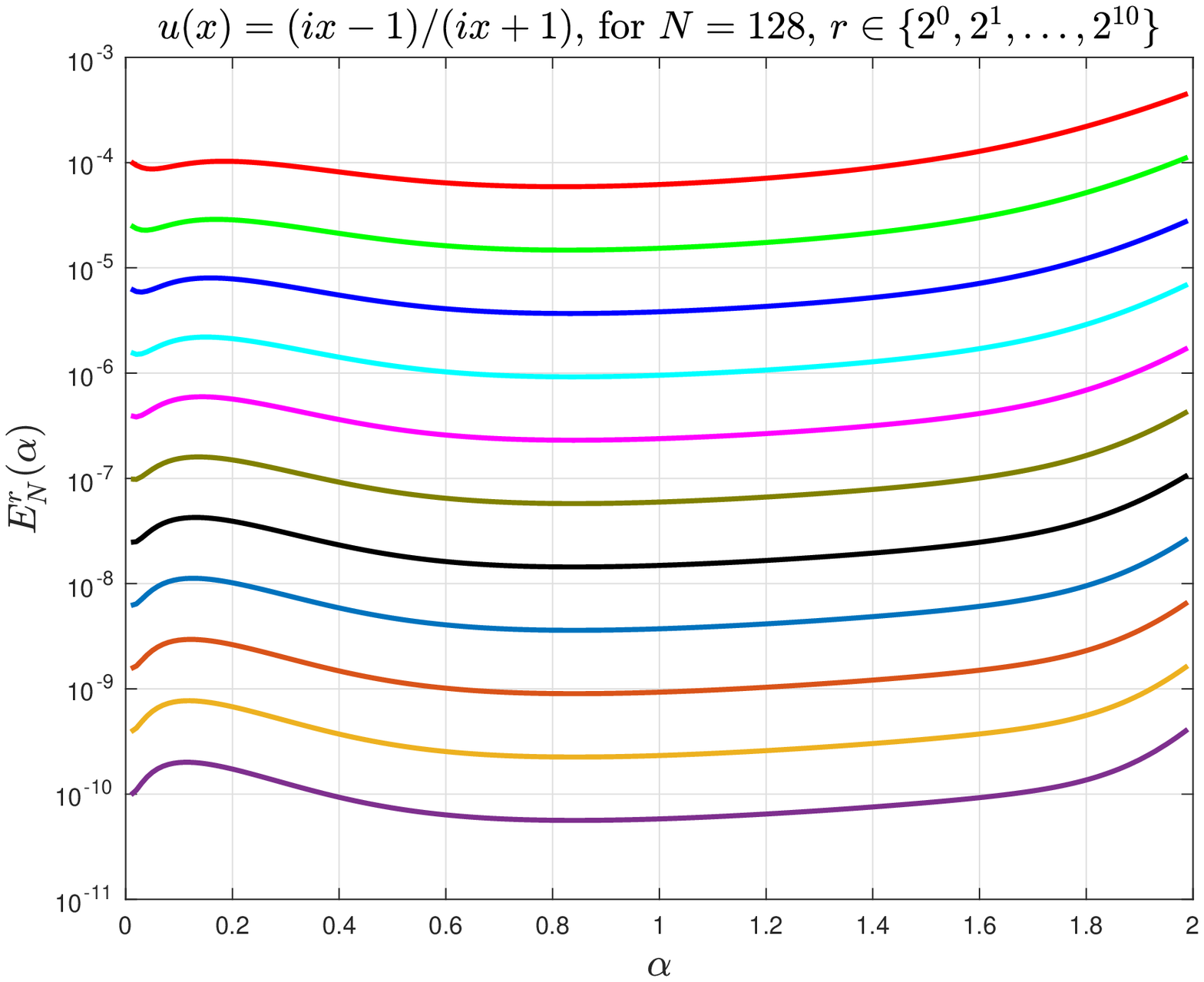}\includegraphics[width=0.5\textwidth, clip=true]{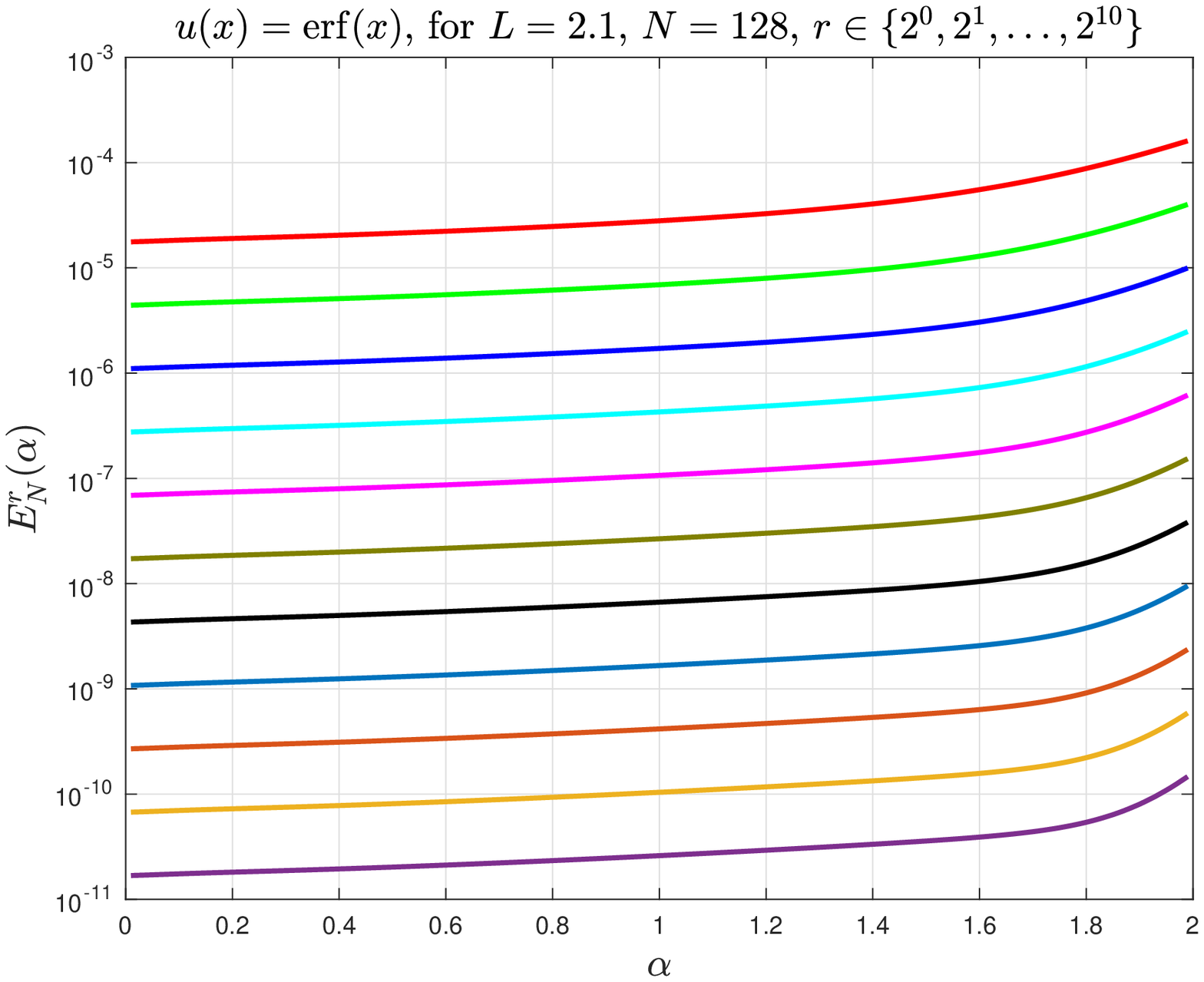}
	\includegraphics[width=\textwidth, clip=true]{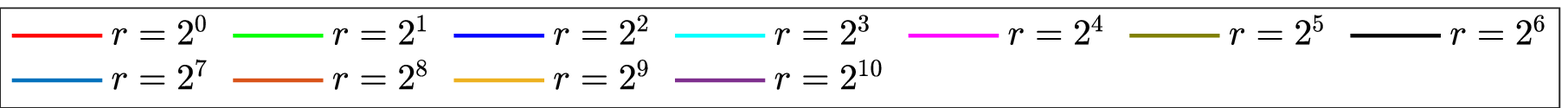}
	\caption{Norms of the errors $E_N^r(\alpha)$, as given by \eqref{e:ENra}, of the numerical experiments in Figure \ref{f:orderofconvergence}, using \eqref{e:deltaa2e2i} (left) and \eqref{e:fraclaperf} (right).}
	\label{f:errors}
\end{figure}

The second order of convergence can be observed not only for values of $r$ that are not powers of $2$, but for arbitrary values of $r$. In Figure \ref{f:errorsallr}, we have plotted  in loglog scale $E_N^{r}(\alpha)$ versus $r\in\{1, \ldots, 1000\}$, for $\alpha=0.1$ and $N\in\{2^1, 2^2, \ldots, 2^{10}\}$. The results for $u(x) = (ix-1)/(ix+1)$ (left) show a set of curves that are asymptotically straight lines with slope $-2$ (recall that we are actually plotting $\log_{10}(E_N^{r}(\alpha))$ versus $\log_{10}(r)$), and the results for $u(x) = \erf(x)$ (right) exhibit that same behavior for $N\in\{2^6, 2^7, 2^8, 2^9, 2^{10}\}$, whereas for $N\in\{2^1, 2^2, 2^3, 2^4, 2^5\}$, from a certain moment on, the norms of the error do not improve as $r$ increases, and a horizontal asymptote appears. This is due to the fact that we are approximating numerically $f(s) = \sin(s)u_{ss}(s) + 2\cos(s)u_{s}(s)$ by means of a pseudospectral method, which, as explained in Section \ref{s:usussnotexplicit}, involves approximating an even extension at $s = \pi$ of $\erf(L\cot(s))$ by means of a Fourier series, which requires an adequately large value of $N$ that may depend on $L$. Then, once that $N$ is large enough, the Fourier series approximates $\erf(L\cot(s))$ up to the accuracy of the machine, and the errors depend exclusively on the quadrature formula that we are using.

\begin{figure}[!htbp]
	\centering
	\includegraphics[width=0.5\textwidth, clip=true]{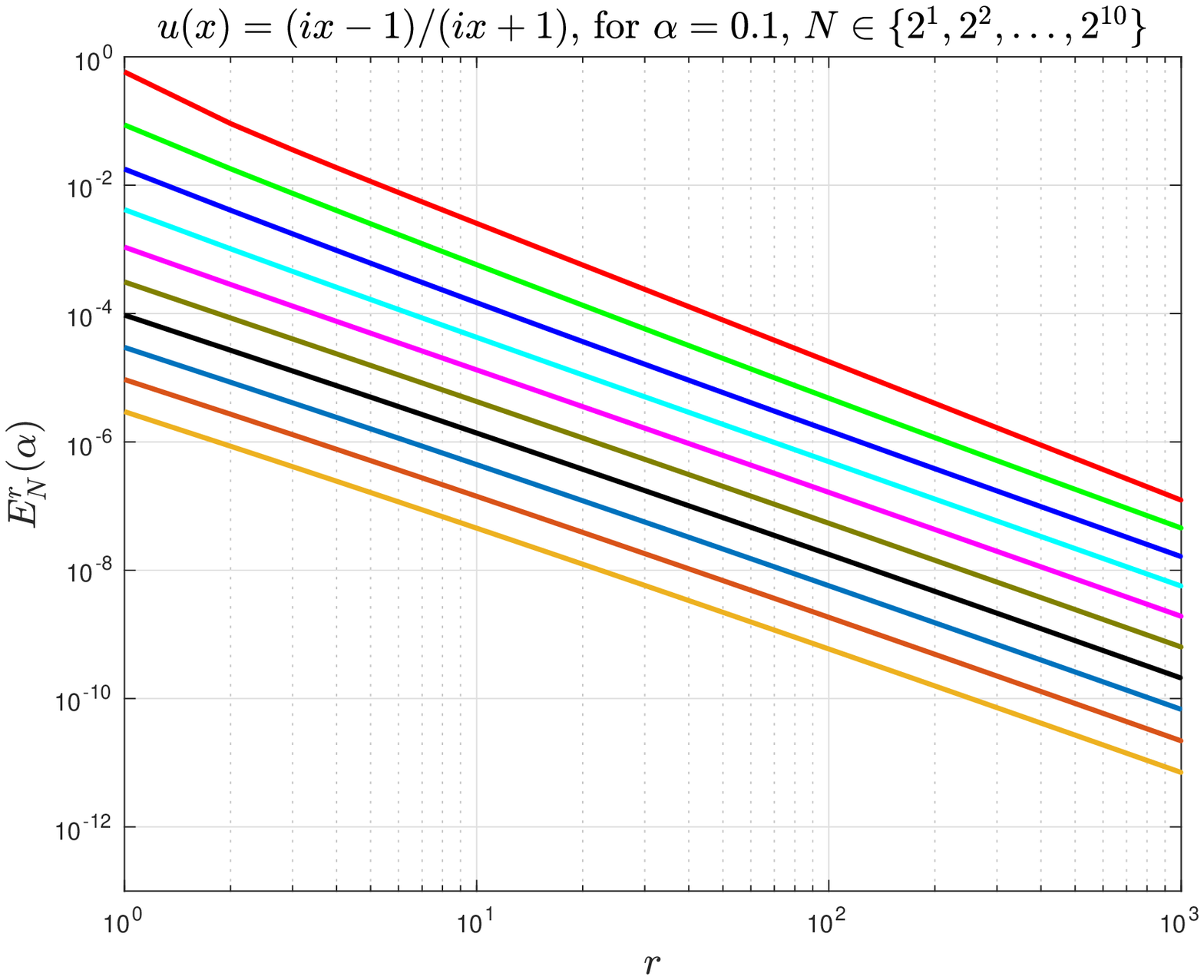}\includegraphics[width=0.5\textwidth, clip=true]{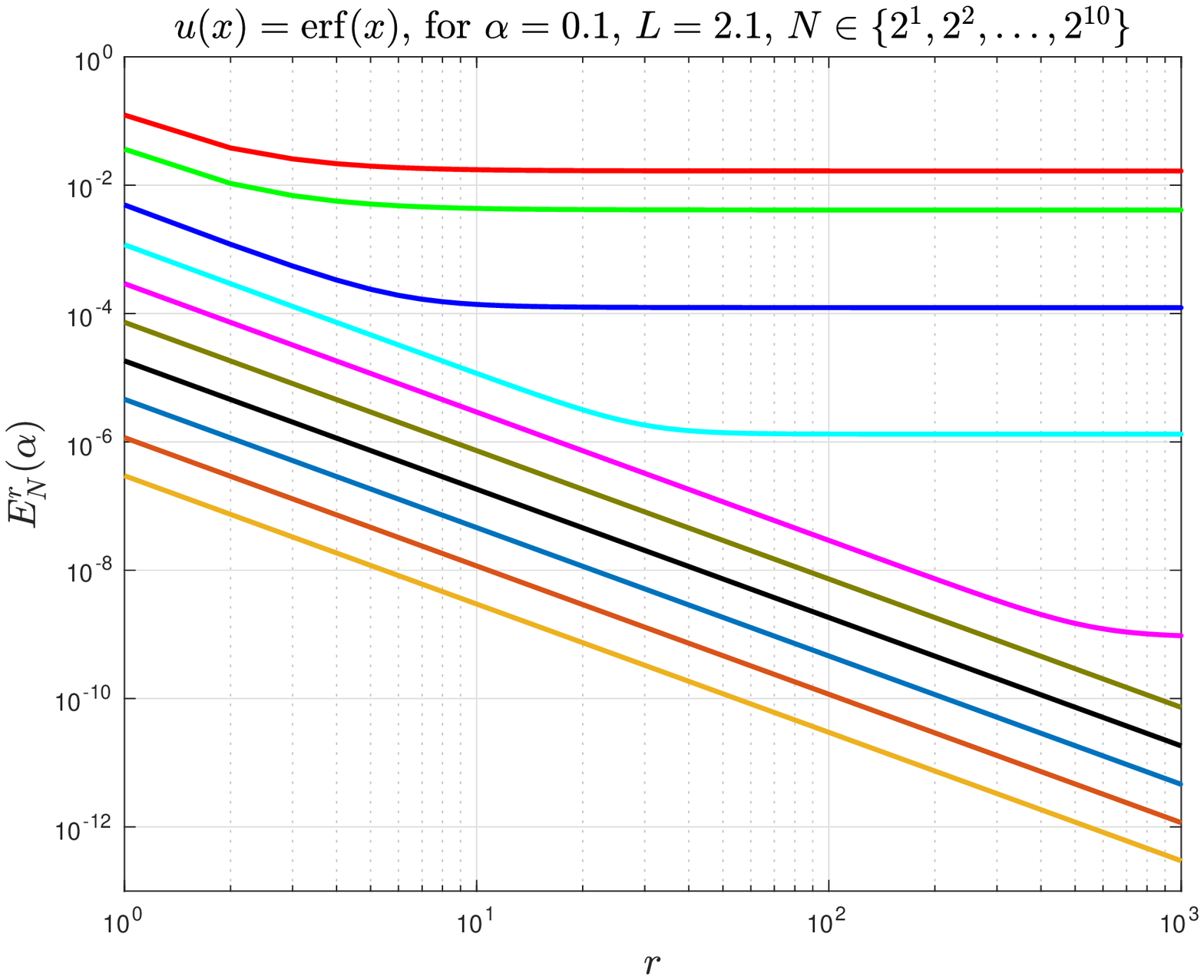}
	\includegraphics[width=\textwidth, clip=true]{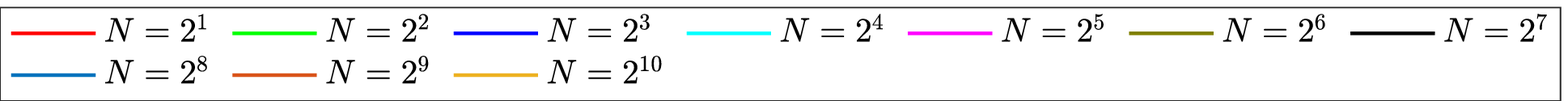}
	\caption{Norms of the errors $E_N^r(\alpha)$, as given by \eqref{e:ENra},  for $r$ and $N$, using \eqref{e:deltaa2e2i} (left) and \eqref{e:fraclaperf} (right).}
	\label{f:errorsallr}
\end{figure}

In Figure \ref{f:errorsallr}, the curves corresponding to $N = 2^{10}$ are below the other curves and, in general, for a fixed $r$, the norms of the errors become smaller as $N$ increases. Therefore, although we have proved theoretically and confirmed numerically the second order of convergence with respect to $r$, a natural question that arises is what order of convergence we have if we fix $r$ and increase $N$. Hence, in Figure~\ref{f:orderofconvergenceallN}, in order to estimate the order of convergence, we have plotted $\log_2(E_{2N}^r(\alpha) / E_N^r(\alpha))$, for $r = 1$, $N\in\{2^1, 2^2, \ldots, 2^{10}\}$, and $\alpha \in\{0.01, \ldots 0.99\}\cup\{1.01, \ldots, 1.99\}$; and in Figure \ref{f:errorsallN}, the corresponding values of $E_N^r(\alpha)$. Although further research is needed here, the results would seem to suggest that $E_{N}^r(\alpha) = \mathcal O(1/N^2)$, i.e., there is also second order of convergence with respect to $N$ when $\alpha\in[0.5, 2)$, whereas, when $\alpha\in(0, 0.5]$, $E_{N}^r(\alpha)  = \mathcal O(1/N^{1.5+\alpha})$.

\begin{figure}[!htbp]
	\centering
	\includegraphics[width=0.5\textwidth, clip=true]{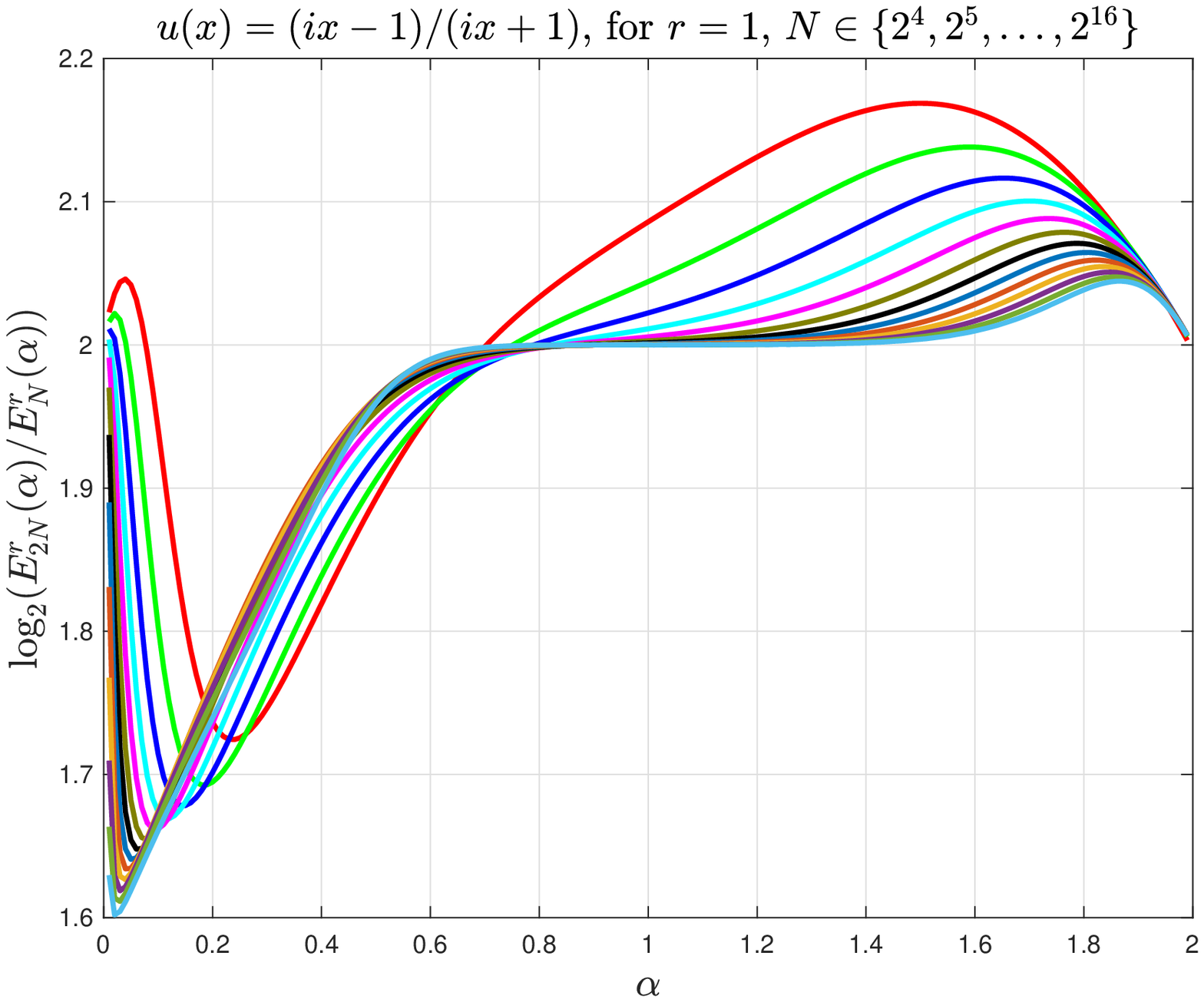}\includegraphics[width=0.5\textwidth, clip=true]{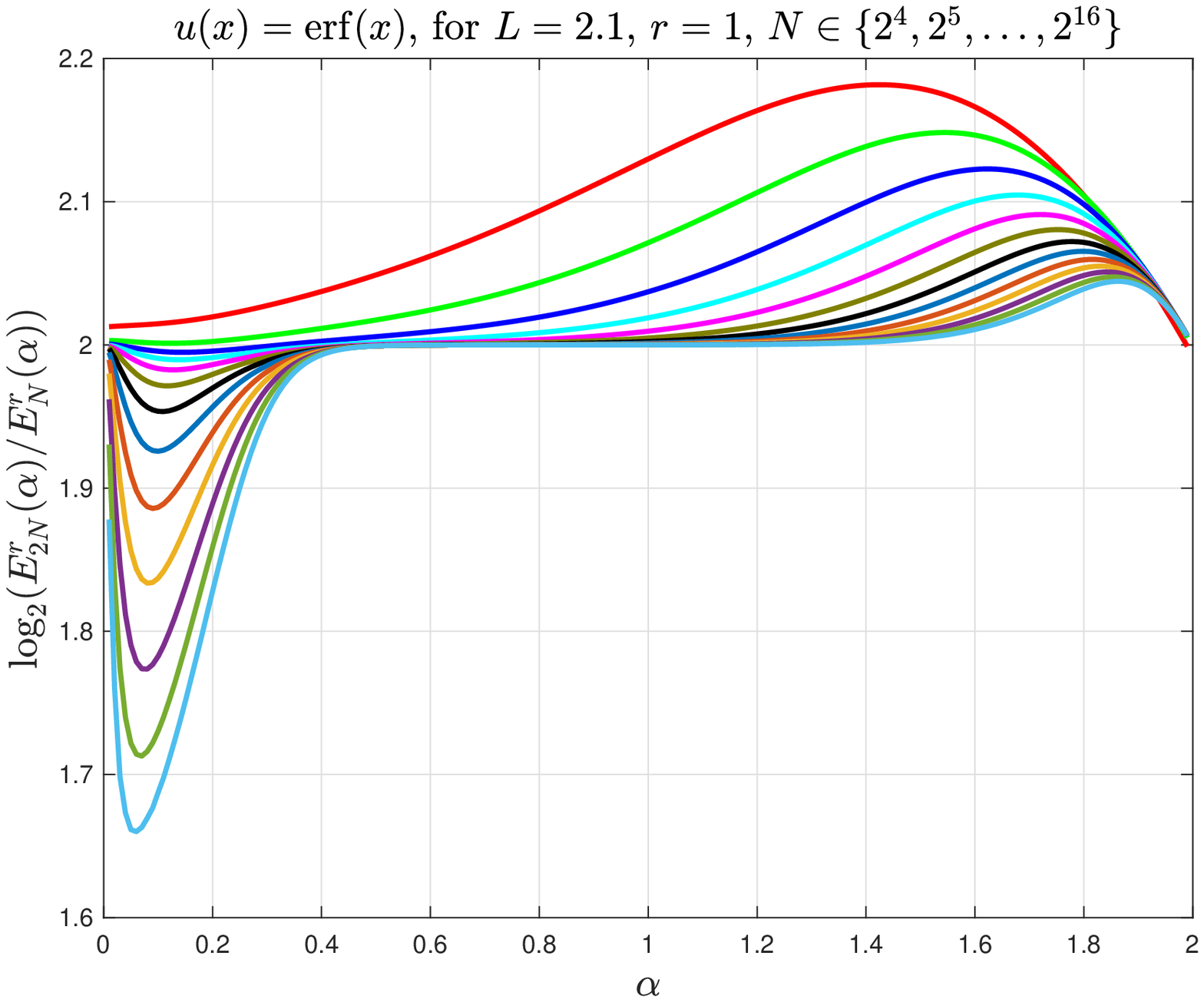}
	\includegraphics[width=\textwidth, clip=true]{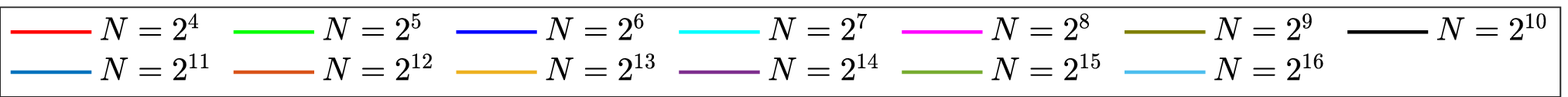}	
	\caption{Numerical measure of the order of convergence of the method. $E_N^r(\alpha)$ is given by \eqref{e:ENra}, using \eqref{e:deltaa2e2i} (left) and \eqref{e:fraclaperf} (right).}
\label{f:orderofconvergenceallN}
\end{figure}

\begin{figure}[!htbp]
	\centering
	\includegraphics[width=0.5\textwidth, clip=true]{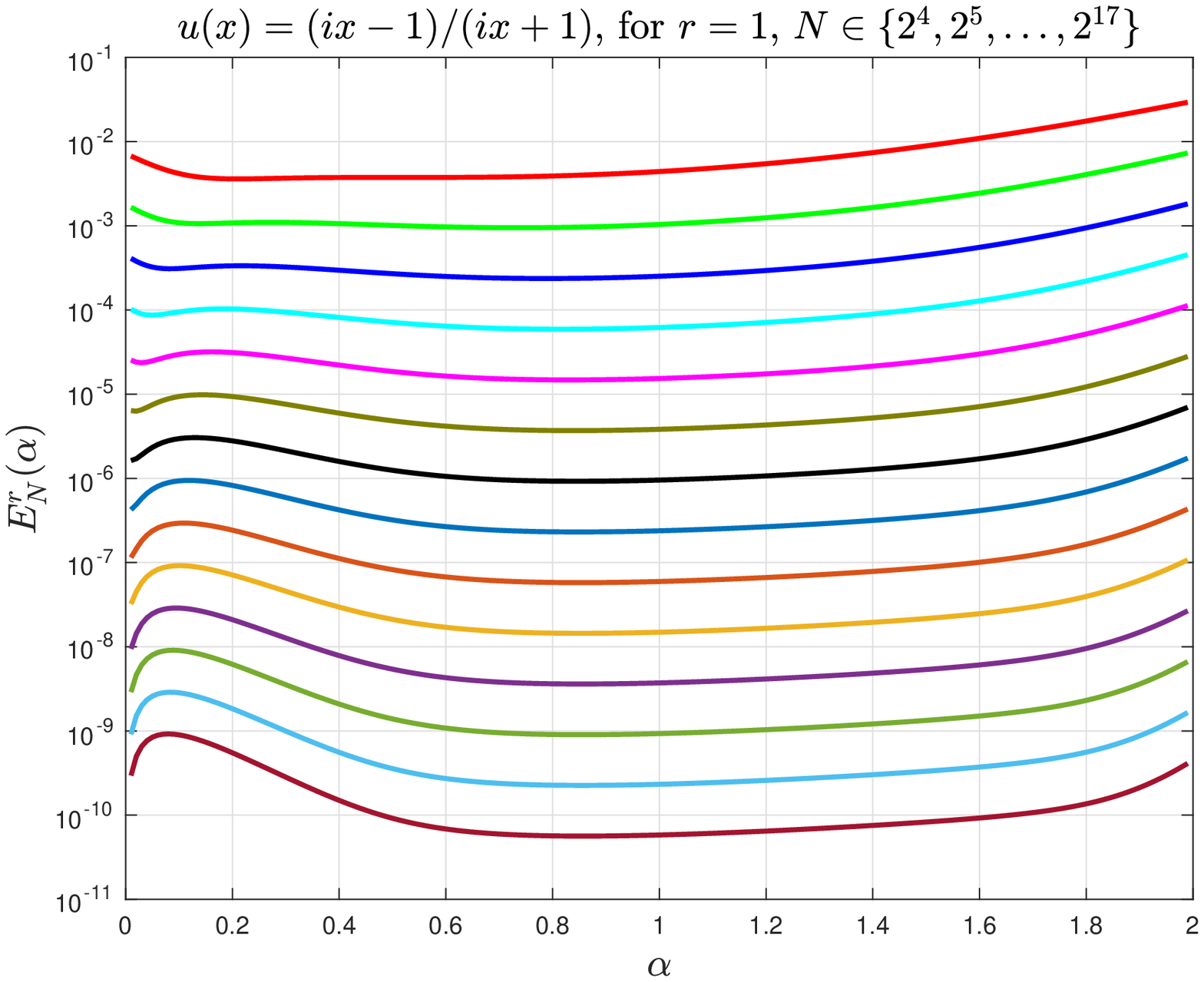}\includegraphics[width=0.5\textwidth, clip=true]{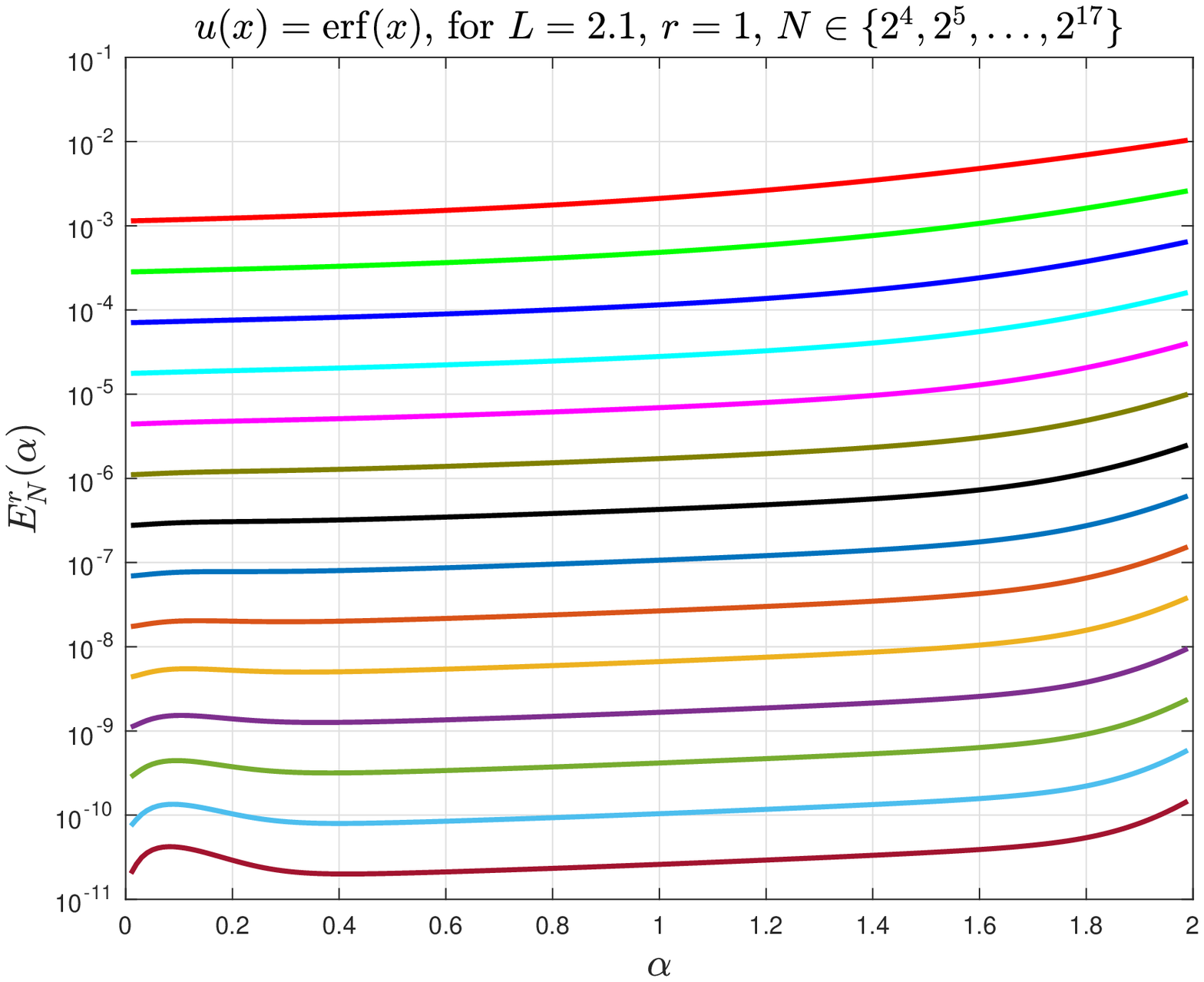}
	\includegraphics[width=\textwidth, clip=true]{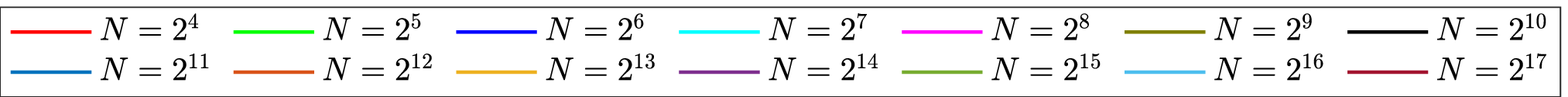}
	\caption{Norms of the errors of the numerical experiments in Figure \ref{f:orderofconvergenceallN}. $E_N^r(\alpha)$ is given by \eqref{e:ENra}, using \eqref{e:deltaa2e2i} (left) and \eqref{e:fraclaperf} (right).}
	\label{f:errorsallN}
\end{figure}

It is a well known fact that some types of differentiation matrices (like Chebyshev differentiation matrices, see, e.g., \cite{gottlieb}) are severely ill-conditioned. Indeed, the condition numbers of the Chebyshev differentiation matrices of order $N$ approximating the first and second derivatives behave respectively as $\mathcal O(N^2)$ and $\mathcal O(N^4)$. Therefore, it is natural to ask ourselves how our numerical method behaves when $2Nr$, which is the number of internal nodes, is extremely large. On the left-hand side of Figure \ref{f:errorshuge}, we have plotted the discrete $L^\infty$ norm of the errors, i.e., $\|(-\Delta)^{\alpha/2}u - (-\Delta)_{\rm num}^{\alpha/2}u\|_\infty$, for $N = 10000019$ (a prime number) and $r = 1$, taking $\alpha\in\{0.01, \ldots 0.99\}\cup\{1.01, \ldots, 1.99\}$. Even if the accuracy greatly depends on $\alpha$, it is at least of the order of $\mathcal O(10^{-10})$ for most values of $\alpha$, and still of the order of $\mathcal O(10^{-9})$ in the worst cases (i.e., when $\alpha$ is close to both $0$ and $2$). Besides, it is particularly remarkable that the errors are of the order of $\mathcal O(10^{-14})$, for a non negligible amount of values of $\alpha$, which confirms again that the implementation of the method is correct. On the other hand, in the right-hand side of Figure \ref{f:errorshuge}, we have plotted the discrete $L^\infty$ norm of the errors, for $N = 128$, and $r = 2^{16}=65536$, and similar conclusions can be drawn, except that the worst results happen only when $\alpha$ is close to $2$. In our opinion, this point is the main caveat of the method and requires further research. However, to the best of our knowledge, as of now, there seems to be no other numerical method capable of approximating the fractional Laplacian at such large amounts of nodes $N$ (and let alone doing it in just a few seconds), so, even in the worst cases, we think that the accuracy is still remarkably high.

\begin{figure}[!htbp]
	\centering
	\includegraphics[width=0.5\textwidth, clip=true]{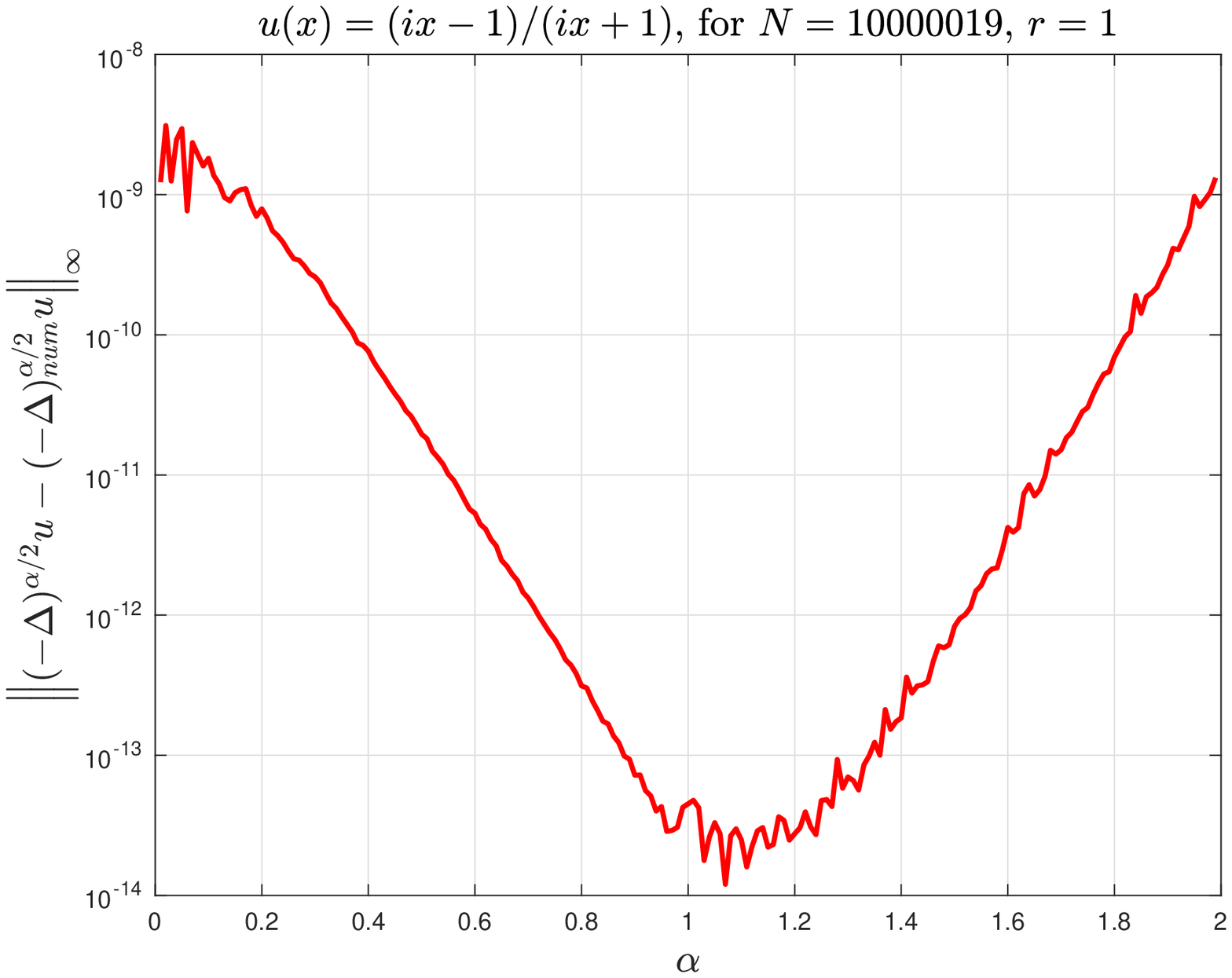}\includegraphics[width=0.5\textwidth, clip=true]{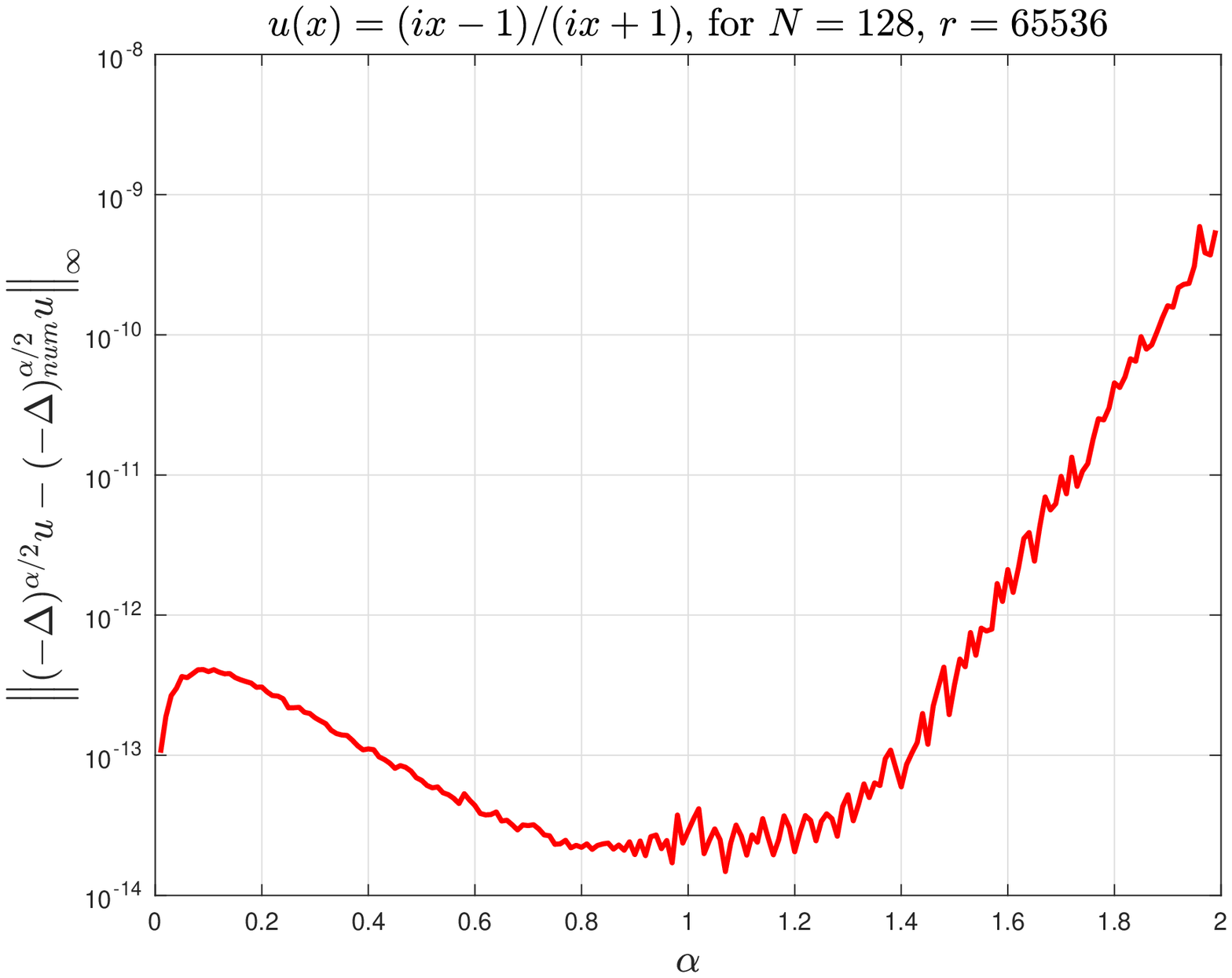}
	\caption{Discrete $L^\infty$ norms of the errors of the numerical experiments, for very large amounts $2Nr$ of internal points.}
	\label{f:errorshuge}
\end{figure}

\subsection{An evolution example}

\label{s:fracSch}

In order to show that the method in this paper can be used to simulate relevant fractional differential equations, we have applied it to the focusing fractional cubic nonlinear Schr\"odinger equation considered in \cite{cayamacuestadelahoz2020} (see the references on this equation in that paper):
\begin{equation}
	\label{e:fracschcub}
	\left\{
	\begin{aligned}
		& i\frac{\partial \psi}{\partial t} = \dfrac{1}2(-\Delta)^{\alpha/2}\psi - |\psi|^2\psi, & & t > 0,
		\\
		& \psi(x, 0) = \psi_0(x).
	\end{aligned}
	\right.
\end{equation}
As in \cite{cayamacuestadelahoz2020}, we have considered $\alpha = 1.99 > 1$, for which \eqref{e:fracschcub} has global existence \cite{GuoHuo2013,GuoSireWangZhao2018}, and have taken $\psi_0(x) = \exp(-x^2)$ as the initial data. We use the classical fourth-order Runge-Kutta method to advance in time, and compute the energy $M(t)$, which is a preserved quantity, to measure the accuracy of the numerical results:
\begin{equation*}
M(t) = \int_{-\infty}^\infty |\psi(x, t)|^2dx = L\int_{0}^\pi\frac{|\psi(L\cot(s), t)|^2}{\sin^2(s)}ds \approx \frac{L\pi}{N}\sum_{j=0}^{N-1}\frac{|\psi(L\cot(s_j), t)|^2}{\sin^2(s_j)},
\end{equation*}
where, in order to approximate $M(t)$, we have applied the composite midpoint rule to $\psi(L\cot(s), t) / \sin^2(s)$ over $s\in[0,\pi]$. This approximation yields spectral accuracy, due to the fact that $\psi(L\cot(s), t) / \sin^2(s)$ is a regular, periodic function over $s\in[0,\pi]$. The values of the other parameters are $N = 4096$, $L = 200$, $\Delta t = 0.01$, $r\in\{2^0, 2^1, \ldots, 2^6\}$. On the upper left-hand side of Figure \ref{f:evolerror}, we have plotted in semilogarithmic scale $|M(t) - M(0)| = |M(t) - \sqrt{\pi/2}|$; the results are coherent with the fact that the approximation of the fractional Laplacian behaves as $\mathcal O(1/r^2)$. Indeed, when we multiply $r$ by $2$, $|M(t) - \sqrt{\pi/2}|$ is approximately divided by $4$, which results in the curves corresponding to the different $r$ being roughly parallel. On the other hand, on the upper right-hand side, we plot the modulus of $\psi(x, t)$; on the lower left-hand side, the real part of $\psi(x, t)$; and on the lower right-hand side, the imaginary part of $\psi(x, t)$; in the three cases, we only show the parts corresponding to $x\in[-150, 150]$.
\begin{figure}[!htbp]
	\centering
	\includegraphics[width=0.5\textwidth, clip=true]{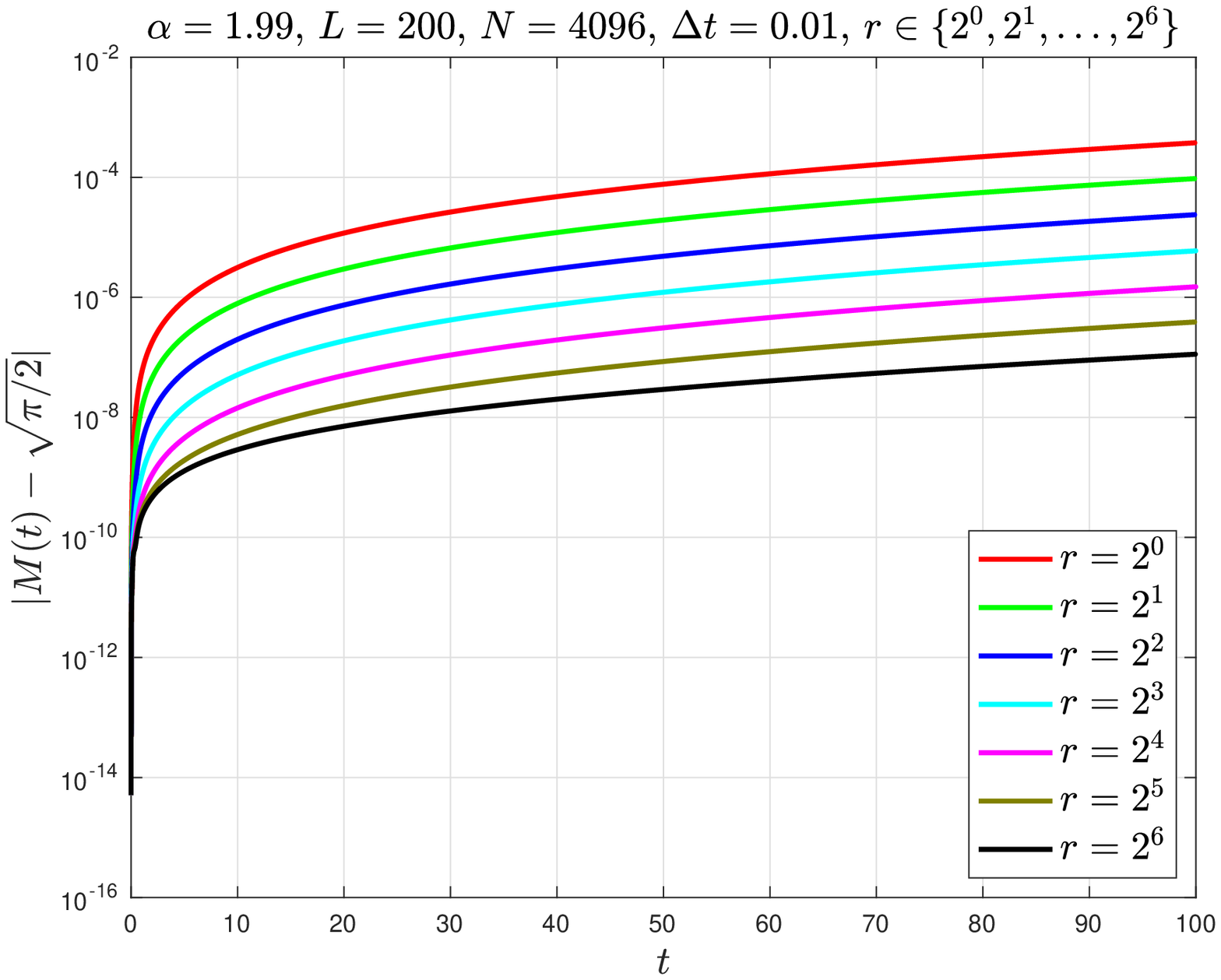}\includegraphics[width=0.5\textwidth, clip=true]{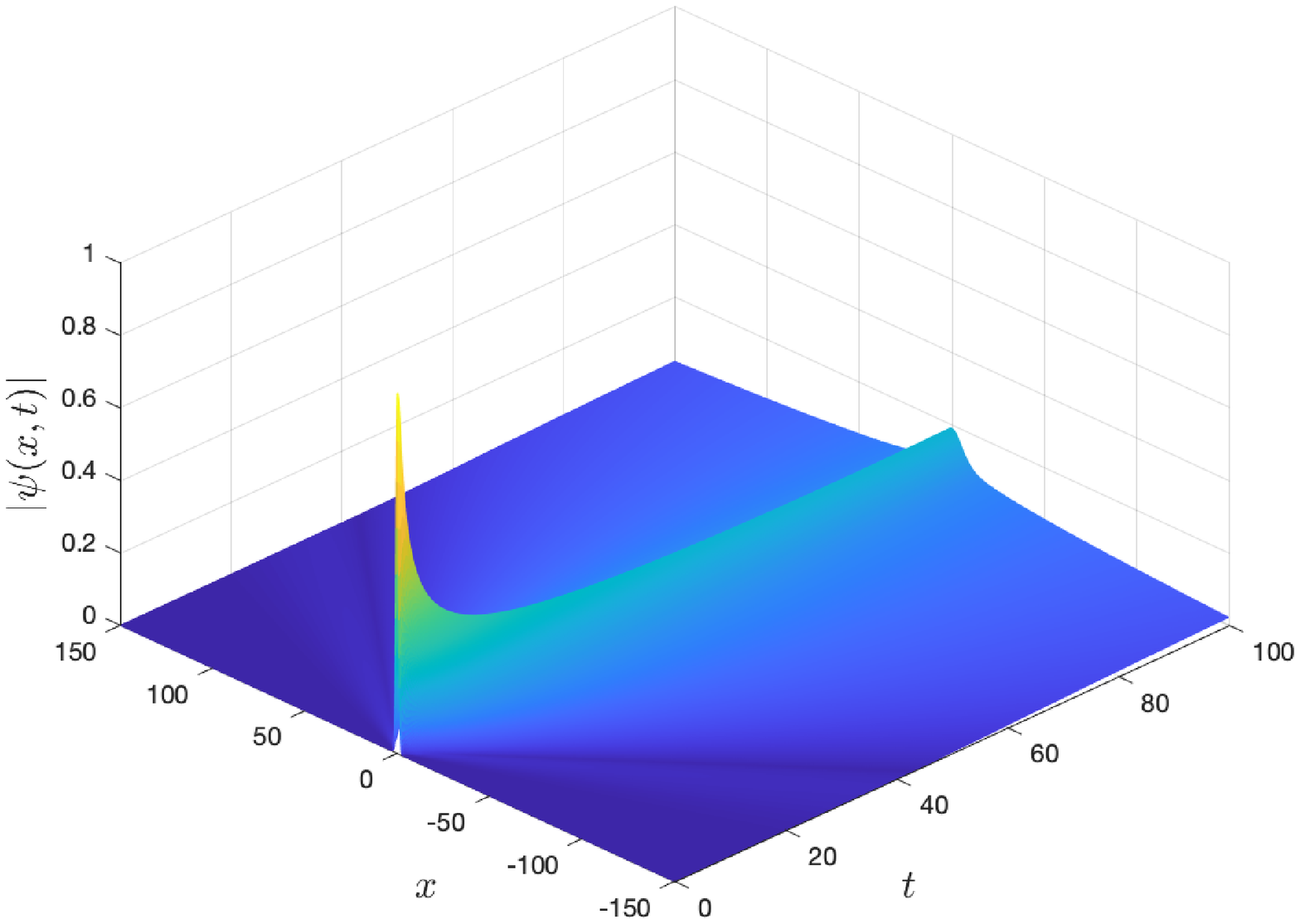}
	\includegraphics[width=0.5\textwidth, clip=true]{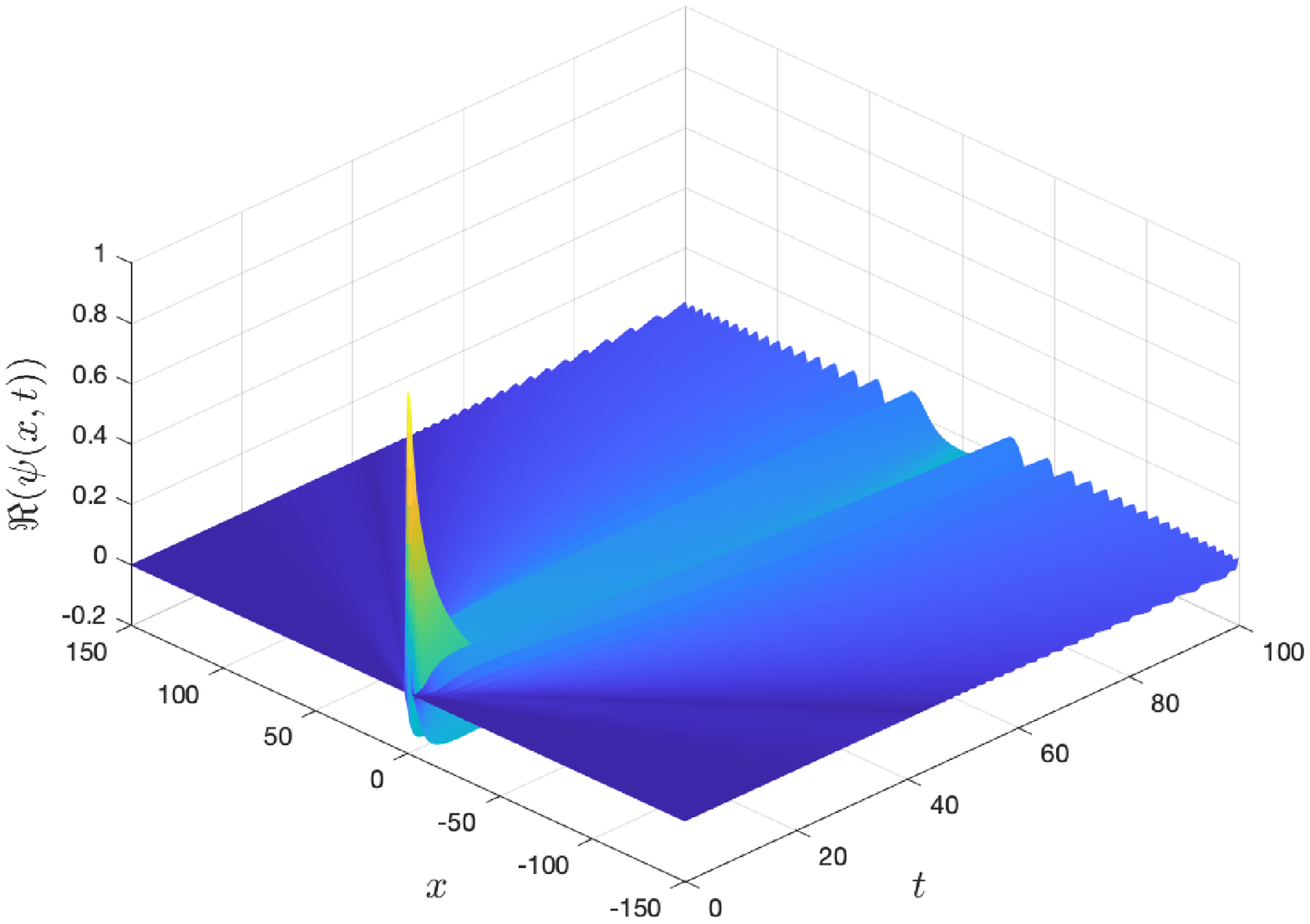}\includegraphics[width=0.5\textwidth, clip=true]{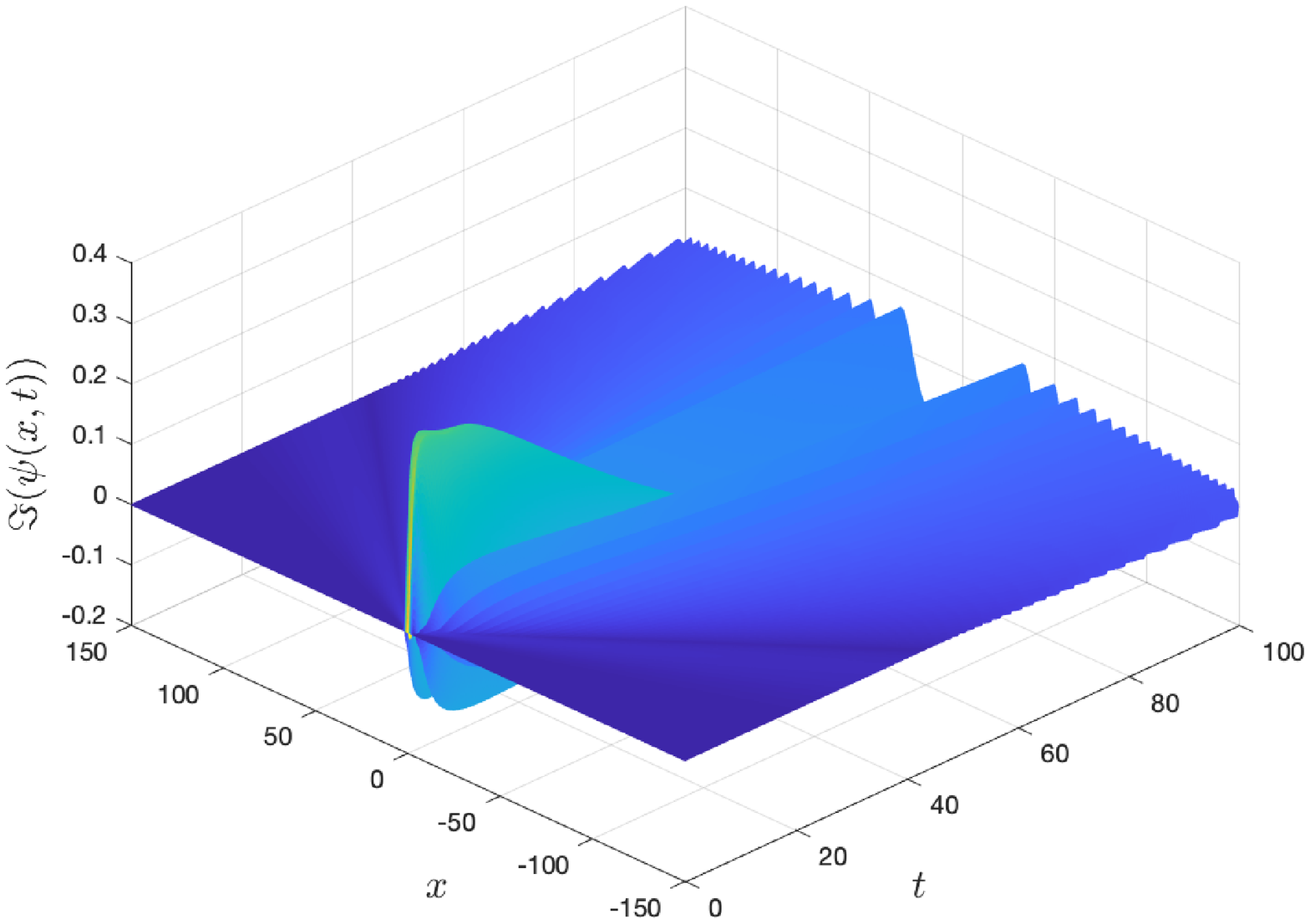}
	\caption{Upper-left: error in the preservation of the energy $E(t)$. Upper-right: modulus of $\psi(x, t)$. Lower-left: real part of $\psi(x, t)$. Lower-right: imaginary part of $\psi(x, t)$.}
	\label{f:evolerror}
\end{figure}

The numerical simulations reveal that the method is very stable for large values of $N$ and not too small values of $\Delta t$, and the choice of $r$ allows to a establish a compromise between execution speed and accuracy. Indeed, given an evolution problem involving the fractional Laplacian, it is possible to consider a first, fast executing simulation with $r = 1$, to get an idea of the evolution, and then increase $r$. For instance, in the experiments in Figure \ref{f:evolerror}, when $r = 2^6 = 64$, the maximum error in the preservation of the energy is equal to $1.1281\times10^{-7}$, and happens at $t = 100$.

\section*{Data Availability Statement}

The authors confirm that the data supporting the findings of this study are available within the article.

\section*{Declarations}

The authors have no competing interests to declare that are relevant to the content of this article.

\end{document}